\documentclass[a4paper,leqno,10pt]{amsart}
\usepackage{hyperref}
\usepackage{mathtools}
\usepackage{tikz}

 \allowdisplaybreaks \numberwithin{equation}{section}
\begingroup
\theoremstyle{plain}
\newtheorem{theorem}{Theorem}[section]
\newtheorem{proposition}[theorem]{Proposition}
\newtheorem{lemma}[theorem]{Lemma}
\newtheorem{corollary}[theorem]{Corollary}
\theoremstyle{definition}
\newtheorem{definition}[theorem]{Definition}
\newtheorem{remark}[theorem]{Remark}

\endgroup

\def \de {\mathrm{d}}
\def \e {\epsilon}

\def \B {\mathbb B}

\def \D {\mathbb D}

\def \Om {\Omega}

\def \la {\lambda}

\def \X {X}
\def \Y {Y}
\def \E {\mathcal D}
\def \D {\mathcal E}
\def \z {\psi}

\title[On Blaschke-Santal\'o diagrams]{On Blaschke-Santal\'o diagrams for the torsional rigidity and the first Dirichlet eigenvalue}
\author{Ilaria Lucardesi}
\address[Ilaria Lucardesi]{Universit\'e de Lorraine, CNRS, IECL, F-54000 Nancy, France}
\email{ilaria.lucardesi@univ-lorraine.fr}
\author{Davide Zucco}
\address[Davide Zucco]{Dipartimento di Matematica ``Giuseppe Peano'', Universit\`a di Torino, Via Carlo Alberto 10, 10123 Torino, Italy}
\email{davide.zucco@unito.it}

\begin{document}
\maketitle

\begin{abstract} We study Blaschke-Santal\'o diagrams associated to the torsional rigidity and the first eigenvalue of the Laplacian with Dirichlet boundary conditions. We work under convexity and volume constraints, in both strong (volume exactly one) and  weak (volume at most one) form. We discuss some topological (closedness, simply connectedness) and geometric (shape of the boundaries, slopes near the point corresponding to the ball) properties of these diagrams, also providing a list of conjectures.
\end{abstract}

{\small


}
\bigskip

\section{Introduction}

Given two shape functionals $\X$ and $\Y$ defined on a class $\mathcal A$ of sets of $\mathbb R^N$, the corresponding \emph{Blaschke-Santal\'o diagram} is the following region of the plane:
$$
 \left\{ (x,y) \in \mathbb R^2\ :\ \text{$\exists \,\Omega\in \mathcal A$ with $x=\X(\Omega)$ and $y=\Y(\Omega)$}\right\},
$$
namely the range of the vector map $\Omega \mapsto (\X(\Omega), \Y(\Omega))$ over the shapes $\Omega$ in $\mathcal A$. For this reason, the diagram is also referred to as \emph{attainable set}. Notice that the map $\Omega \mapsto (\X(\Omega), \Y(\Omega))$ is in general not injective, since different shapes could be associated to the same point.

Typically, the class $\mathcal A$ encodes some constraints, that prevent the diagram to be trivial (e.g., the whole plane, a whole quadrant, a line, or a half-line). They can be either bounds, or prescribed values, for some quantities (such as volume, perimeter, diameter, inradius), or geometric and topological restrictions (such as convexity or simple connectedness).

The goal is to identify the attainable set, in particular its boundary and the shapes {associated to the points} on it. A complete description would amount to characterize the relations between the shape functionals $X$ and $Y$, by means of (optimal) upper and lower bounds in terms of the boundary points of the diagram and the associated shapes. Since shape functionals and their bounds appear in several mathematical areas (e.g., Poincar\'e inequalities in functional analysis), Blaschke-Santal\'o diagrams are very useful tools, and the literature on the subject is quite vast (see for instance \cite{AH, VDBBV, BHT,Bla, BNP, BBF,  MZ, Santa} and more recently \cite{VDBBP,FL}). As it appears from the literature, the theoretical analysis, even if very fine, is in general not enough for an accurate description, and some aspects remain unsolved. Often, conjectures are supported by numerical simulations.

In this paper we study the Blaschke-Santal\'o diagram corresponding to the first eigenvalue of the Dirichlet Laplacian and to the torsional rigidity, under volume and convexity constraints. 
Given an open bounded set $\Omega$ of $\mathbb R^N$, the \emph{first Dirichlet eigenvalue} $\la_1(\Om)$ and the \emph{torsional rigidity} $T(\Omega)$ are defined as follows:
\begin{equation}\label{minmax}
\lambda_1(\Omega) := \min_{u\in H^1_0(\Omega)\setminus\{0\}} \frac{\int_{\Omega}|\nabla u(x)|^2dx}{\int_{\Omega}|u(x)|^2dx} \quad
\text{ and } \quad
T(\Omega):=\max_{u\in H^1_0(\Omega)\setminus\{0\}}\frac{\big(\int_{\Omega} u(x) dx \big) ^ 2}{\int_{\Omega}|\nabla u(x)|^2dx}.
\end{equation}
It is well-known that these minimum and maximum are achieved, respectively, by the so-called {\it first eigenfunction} $\varphi_\Omega$ and {\it torsion function} $w_\Omega$. These functions are unique up to a multiplicative constant, therefore, in this paper we choose to work with the first eigenfunction $\varphi_\Omega$ normalized in $L^2(\Omega)$, such that $\lambda(\Omega)=\int_\Omega|\nabla \varphi_\Omega|^2$,  and with the torsion function $w_\Omega$ such that $T(\Omega)=\int_\Omega w_\Omega=\int_\Omega|\nabla w_\Omega|^2$.
Notice also that they are weak solutions in $\Omega$ of the following PDEs:
$$
-\Delta \varphi_\Omega = \lambda_1(\Omega) \varphi_\Omega \quad\text{ and } \quad  -\Delta w_\Omega = 1,
$$
with zero boundary condition on $\partial \Omega$. 
Our aim is to characterize the Blaschke-Santal\'o diagram when $\X=\lambda_1$ and $\Y=T^{-1}$ over the class $\mathcal A$ of convex sets with unit volume: 
$$
\E:=
\left\{ (x,y)\in \mathbb R^2\ :\ \text{$\exists\, \Omega \subset \mathbb R^N$ convex, $|\Omega|=1$, with $x=\lambda_1(\Omega)$ and $y=T(\Omega)^{-1}$}\right\},
$$
where $|\cdot|$ denotes the $N$-dimensional Lebesgue measure. The choice of pairing $\lambda_1$ with the inverse of $T$ (instead of $T$) is natural: as it is clear from \eqref{minmax}, they share many properties, e.g., they are both monotonically decreasing with respect to set inclusion and they are homogeneous with negative indeces.
Actually, the volume constraint can be removed, up to enclosing it into the shape functionals: since $\lambda_1$ is $-2$-homogeneous and $T$ is $(N+2)$-homogeneous, we have
$$
\E=\{ (x,y)\in \mathbb R^2\ :\ \exists\, \Omega \subset \mathbb R^N\ \hbox{convex},\ x=\lambda_1(\Omega)|\Omega|^{2/N},\ y=T(\Omega)^{-1}|\Omega|^{(N+2)/N}\}.
$$
In this paper, we also address the variant with volume constraint in a \emph{weak form}:
$$
\D:=\left\{ (x,y)\in \mathbb R^2\ :\ \text{$\exists\, \Omega \subset \mathbb R^N$ convex, $|\Omega|\leq1$, with $x=\lambda_1(\Omega)$ and  $y=T(\Omega)^{-1}$} \right\},
$$
which clearly contains $\E$.

The classical inequalities (see, e.g. \cite{H, Heditor, PS})
\begin{align}
& T(\Omega) |\Omega|^{-(N+2)/N}\leq T(B) |B|^{-(N+2)/N} \qquad & (Saint-Venant)\label{SV}
\\
& \lambda_1(\Omega) |\Omega|^{2/N} \geq \lambda_1(B) |B|^{2/N}  \qquad & (Faber-Krahn)\label{FK}
\\
& T(\Omega) \lambda_1(\Omega) \leq  |\Omega| \qquad &  ({\hbox{{\it P\'olya}}})\label{P}
\\
& T(\Omega)^{2/(N+2)} \lambda_1(\Omega) \geq    T(B)^{2/(N+2)} \lambda_1(B) \qquad & (Kohler-Jobin)\label{KJ}
\end{align}
valid  for every open set $\Omega$ of $\mathbb R^N$ and for every ball $B$ of $\mathbb R^N$, define, in a natural way, 
a region $\mathcal R$ including 
the diagrams $\E$ and $\D$:
\begin{equation}\label{DeU}
\E\subset \D \subset \mathcal R:= \{(x,y)\in \mathbb R^2\ :\ y \geq  T(\B)^{-1},\  x \geq \lambda_1(\B),\  y\geq x, \ y\leq  	c_\B \, x^{(N+2)/2}\},
\end{equation}
where $c_\B\coloneqq1/[T(\B)\lambda_1(\B)^{(N+2)/2}]$ and $\B$ denotes the $N$ dimensional ball of unit volume. To fix the ideas, in Figure~\ref{figura1}, we plot the region $\mathcal R$ for $N=2$, where $ \lambda_1(\B)=\pi j_{0,1}^2 \sim 18$ ($j_{0,1}$ is the first zero of the Bessel function $J_0$), $T(\B)^{-1}=8\pi \sim 25$, and $c_\B \sim 0.077$.

The \emph{Kohler-Jobin curve} $\Gamma_{\mathbb B}:=\{y=c_\B \, x^{(N+2)/2} \ : \ x\geq \lambda_1(\mathbb B)\}$, corresponds to sets of volume less than or equal to 1 realizing the equality in the Kohler-Jobin inequality \eqref{KJ}, namely each point of this curve is \emph{uniquely} associated to a ball of volume less than or equal to one. 
The constant 1 in front of the volume in the P\'olya inequality \eqref{P} is optimal for generic sets, in the sense that it cannot be lowered: this is shown in \cite{VDBFNT} by taking a suitable sequence of perforated domains (\emph{\`a la} Cioranescu-Murat), whose first Dirichlet eigenvalues go to $+\infty$, whereas their torsional rigidities go to $0$. In other words, the bisector $y=x$ is approached asymptotically, by some points of the diagram whose horizontal and vertical components diverge.
These results, together with the fact that balls realize the equalities in \eqref{SV} and \eqref{FK}, imply that
$\Gamma_{\mathbb B}$ is the only piece of the boundary of $\mathcal R$ belonging to $\D$; this is a quite rough information. If we restrict ourselves to the set $\E$, the situation is even worse:  the only point of $\partial \mathcal R$ in the diagram $\E$ is the vertex $\mathbb V:=(\lambda_1(\B), T(\B)^{-1})$. However, for convex sets, there holds a \emph{reverse P\'olya inequality} $\lambda_1(\Omega) T(\Omega) \geq C_N |\Omega| $ for some dimensional constant  $C_N>0$ (this is explicitly determined in \cite[Theorem 1.4, formula (1.7)]{VDBFNT}). This translates into the following bound:
\begin{equation}\label{ubE}
\E \subset \{ (x,y) \in \mathbb R^2\ :\ y \leq x / C_N  \},
\end{equation}
which indeed states that the diagram is bounded from above by a linear function.

\begin{figure}[t]                                         
\begin{tikzpicture}[domain=0:10, scale=0.5]   

  \draw[->] (-0.2,0) -- (10,0) node[below] {$x$};
  \draw[->] (0,-0.3) -- (0,10) node[left] {$y$};
  \draw[draw=gray!50!white,fill=gray!50!white]
  plot[domain=2.5:9.3] (\x,{(\x)}) --
  plot[domain=3.5:1.8] (\x,{(0.77*(\x)^2)}); 
  \draw[-] (1.8,2.5) -- (2.5,2.5);
  \draw[domain=2.5:9.3]  plot (\x,{(\x)}) node[above] {\tiny $y=x$};
  \draw[domain=1.8:3.5]  plot (\x,{(0.77*(\x)^2)}) node[above] {\tiny $\qquad \quad y=c_\B x^{(N+2)/2}$};
  \draw[-, dashed, very thin] (1.8,0) node[below] {\tiny $\la_1(\B)$} -- (1.8,10) ;
  \draw[-, dashed, very thin] (0,2.5) node[left] {\tiny $T(\B)^{-1}$} -- (10,2.5);
  \fill[black] (1.8,2.5) circle (2.5pt) node[below left] {\tiny $\mathbb V$};

\end{tikzpicture}
\caption{The region $\mathcal R$ containing the Blaschke-Santal\'o diagrams $\D$ and $\E$.}\label{figura1}
\end{figure}
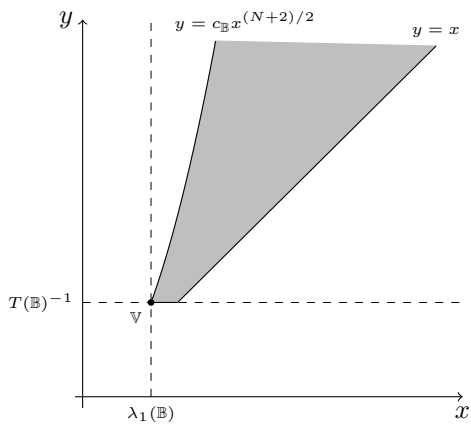

As shown above, the available results relating $\lambda_1$ and $T$ only allow to give some bounds on the diagrams. The challenging problem of completing the description motivates our study. In the following two theorems we summarize our results on the diagrams, under volume constraint in the strong and weak form.

\begin{theorem}[The diagram $\E$]\label{thmE}
There hold the following properties.
\begin{itemize}
\item[1.] \emph{(Topology).} The diagram $\E$ is closed, connected by arcs, and $\mathbb R^2\setminus \overline\E$ has only one unbounded connected component.
\item[2.] \emph{(Boundary).} The unbounded connected component of $\partial \E$ is the union of two curves 
which meet at the vertex $\mathbb V:=(\lambda_1(\B), T(\B)^{-1})$ and diverge to $+\infty$ as $x\to+\infty$. 
\item[3.] \emph{(Tangents at the vertex).} In dimension 2 the maximal and minimal slopes $\gamma^\pm$ at $\mathbb V$ (see \eqref{gammapm} for the definition) satisfy
\begin{equation}\label{eq.slopes}
{\gamma^+}=\frac{16}{j_{0,1}^2} \qquad\text{and}\qquad 0\leq\gamma^-\leq\frac{32}{j_{0,1}^2(j_{0,1}^2-2)},
\end{equation}
respectively.
\end{itemize}
\end{theorem}

\begin{theorem}[The diagram $\D$]\label{thmD}
There hold the following properties.
\begin{itemize}
\item[1.] \emph{(Topology).} The diagram $\D$ is closed, simply connected, convex in the $x$-direction and convex in the $y$-direction.
\item[2.] \emph{(Boundary).} Its boundary $\partial \D$ is the union of two curves which meet only at the vertex $\mathbb V:=(\lambda_1(\B), T(\B)^{-1})$ and diverge to $+\infty$ as $x\to+\infty$. The boundary above the diagram $\D$ is the Kohler-Jobin curve $\Gamma_{\mathbb B}:=\{y=c_\B \, x^{(N+2)/2} \ : \ x\geq \lambda_1(\mathbb B)\}$, where $c_\B\coloneqq1/[T(\B)\lambda_1(\B)^{(N+2)/2}]$. The boundary below the diagram is a continuous increasing curve.
\item[3.] \emph{(Measure of shapes).} The measure of a shape $\Omega$ associated to a point $(x,y)\in\D$ is bounded below by 
$$
|\Omega| \geq \max\left\{\frac{\lambda_1(\B)}{x}, \left(\frac{1}{T(\B)y} \right)^{\frac{N}{N+2}}\right\}.
$$
\item[4.] \emph{(Tangents at the vertex).} In dimension 2 the maximal and minimal slopes at $\mathbb V$ are the same $\gamma^\pm$ found in \eqref{eq.slopes} of Theorem \ref{thmE}.
\end{itemize}
\end{theorem}

\medskip

Recently, the same pair of shape functionals has been considered in \cite{VDBBP}, in which the authors investigate upper and lower bounds for functionals of the form $\lambda_1(\Omega)T(\Omega)^q$.

The plan of the paper is the following: in Section \ref{sec-prel} we fix some notation and we recall some tools of shape optimization, such as the Hausdorff metric, the continuous Steiner symmetrization, Minkowski sums, and shape derivatives. For the benefit of the reader, some of the proofs are postponed to the Appendix (Section \ref{sec-app}). In Section \ref{sec-E}, we study the diagram $\E$: the statement 1 of Theorem \ref{thmE} is proved in Propositions \ref{chiusoecpa} and \ref{noholes1}, the statement 2 in Propositions \ref{prop-Lpiu} and \ref{prop-Lmeno}, and the statement 3 in Proposition \ref{prop-valorislopes}.
Then, in Section \ref{sec-D}, we impose the inequality sign in the volume constraint, describing the diagram $\D$:  the statement 1 of Theorem \ref{thmD} is proved in Proposition \ref{prop-Dcvx}, the statement 2 in Propositions \ref{prop-L2} and \ref{prop43}, the statement 3 in Remark \ref{rem41}, and the statement 4 in Corollary \ref{corollario}.
The study led us to address some very deep questions, whose answer is beyond the scope of the present paper. We list them at the end of the paper, in Section \ref{sec-open}, together with some comments and conjectures.

\section{Preliminaries}\label{sec-prel}
In this section we fix some notation by recalling known facts that will be useful in the sequel.
In the proofs and in the technical parts, we write $\X(\Omega)$ and $\Y(\Omega)$ for $\lambda_1(\Omega)$ and $T(\Omega)^{-1}$, respectively. We say that a point $(x,y)$ in a Blaschke-Santal\'o diagram is associated to a set $\Omega$ when $X(\Omega)=\lambda_1(\Omega)=x$ and $Y(\Omega)=T(\Omega)^{-1}=y$.

\subsection{Hausdorff metric}\label{hausdorff.metric}

We endow the class of open convex sets with the \emph{Hausdorff complementary metric} (in short Hausdorff metric): the Hausdorff distance of two open sets is defined through the Hausdorff distance of their complements, which are closed sets (see \cite{HP}).
In the paper we will need the following well-known result.

\begin{lemma}\label{lemma.h}
Let $\{\Omega_n\}$ be a sequence of convex sets of $\mathbb R^N$ such that $\sup_n |\Omega_n|<+\infty$ and $\sup_n \lambda_1(\Omega_n)<+\infty$.
Then the following facts hold.
\begin{itemize}
\item[-] \emph{(Compactness).} There exists a convex set $\Omega$ of $\mathbb R^N$ such that, up to subsequences (that we do not relabel), $\Omega_n$ converges to $\Omega$ in the Hausdorff metric.
\item[-] \emph{(Continuity).} For the subsequence of the previous item there hold
\[
\lim_{n\to\infty} |\Omega_n|=|\Omega|,\quad  \lim_{n\to\infty} \lambda_1(\Omega_n)=\lambda_1(\Omega),\quad  \lim_{n\to\infty} T(\Omega_n)=T(\Omega).
\]
\end{itemize}

\end{lemma}
\begin{proof}
First notice that by \eqref{KJ} we also have that $\sup_n T(\Omega_n)^{-1}<+\infty$.
To prove the lemma it is sufficient to show that $\sup_n \text{diam}(\Omega_n)<+\infty$, where $\text{diam}(\Omega_n)$ denotes the diameter of $\Omega_n$. If so, the family $\{\Omega_n\}$ turns out to be uniformly bounded and then compactness and continuity of volume, first eigenvalue, and torsional rigidity are well known, see \cite[Theorems 2.3.15 and 2.3.17]{H} or also \cite{BB,HP}.

In order prove that the family $\{\Omega_n\}$ is uniformly bounded, we fix a set $\Omega_n$ and use (a weak version of) the \emph{Hersh-Protter inequality} \cite{Hersh, P}, which provides a lower bound on the first eigenvalue $\lambda_1(\Omega_n)$ of a convex set in terms of a power of its inradius $\rho(\Omega_n)$:
\begin{equation}\label{eq.HP}
\lambda_1(\Omega_n) \geq \frac{\pi^2}{4\rho(\Omega_n)^2}.
\end{equation}
Moreover, by convexity, diameter and inradius give a lower bound on the volume: 
indeed, by considering the convex hull of a ball with radius $\rho(\Omega_n)$ and of a segment with length $\mathrm{diam}(\Omega_n)$, both contained into the convex set $\Omega_n$, we infer that 
\begin{equation}\label{stima-diam}
K_N\rho(\Omega_n)^{N-1} \mathrm{diam}(\Omega_n)\leq |\Omega_n|,
\end{equation}
where $K_N$ is a dimensional constant independent of $n$. By combining \eqref{eq.HP} with \eqref{stima-diam}, and taking the supremum with respect to $n$, we finally get a uniform bound on the diameters of the family $\{\Omega_n\}$, thanks to the hypothesis of the lemma.
\end{proof}

\subsection{Three continuous paths}\label{sec-CSS}

In this section we introduce three kinds of continuous paths joining pairs of points in the diagrams. Roughly speaking, starting from a continuous deformation of sets $t\mapsto \Omega_t$ from $\Omega_0$ to $\Omega_1$, we end up with a curve $t\mapsto (\X(\Omega_t), \Y(\Omega_t))$ in the diagram.

\medskip

The first deformation that we consider is the homotecy: given a bounded open set $\Omega$ of volume less than or equal to 1, all the homotecies $t\Omega$, with $0<t\leq |\Omega|^{-1/N}$, have still volume less than or equal to 1. For $t\in (0,1]$ we have compressions, whereas for $t\in [1, |\Omega|^{-1/N}]$ we have dilations. In particular, the set
$$
\Gamma_\Omega\coloneqq \left \{  (\X(t\Omega), \Y(t\Omega))\ :\ t\in \left(0,|\Omega|^{-1/N}\right]  \right \}
$$ 
is contained into the diagram $\D$.
Notice that, in view of the homogeneity of $\X(\cdot)$ and $\Y(\cdot)$ (of order $-2$ and $-(N+2)$, respectively), such set is a smooth curve whose explicit formula is
$$
\Gamma_\Omega=\left\{ y= \frac{\Y(\Omega)}{\X(\Omega)^{(N+2)/2}}x^{(N+2)/2}  \ :\ x\geq |\Omega|^{2/N}\lambda_1(\Omega) \right\}.
$$
Similarly, we define the portions of the curve associated to homotecies $t\Omega$, $0<t_1\leq t\leq t_2\leq |\Omega|^{-1/N}$:
$$
\Gamma_\Omega(t_1,t_2)\coloneqq \{ (\X(t\Omega), \Y(t\Omega))\ :\ t\in [t_1,t_2]  \}.
$$

\begin{remark}\label{rem-KJ} Notice that $\Gamma_{\Omega}$ is a portion of the curve $y=c x^{(N+2)/2}$, which is superlinear and passes through the origin. The coefficient $c$ depends on the shape and  is the same for homotetic sets, since $c={\Y(\Omega)}{\X(\Omega)^{-(N+2)/2}}$ is scale invariant. Moreover, it is easy to see that if $\X(\Omega_1)=\X(\Omega_2)$ and $\Y(\Omega_1)<\Y(\Omega_2)$, then $\Gamma_{\Omega_1}$ lies below $\Gamma_{\Omega_2}$.

When $\Omega=B$ is a ball, $\Gamma_B$ is nothing but the Kohler-Jobin curve $\Gamma_{\mathbb B}$. For this reason, in the sequel we will refer to these curves as of {\it Kohler-Jobin type}.
\end{remark}

\medskip

The second deformation that we introduce is the so-called {\it continuous Steiner symmetrization}.
Roughly speaking, as the name itself suggests, it is the continuous version of the ``classic'' Steiner symmetrization. For a detailed presentation, see \cite{B, BH} and the very recent paper \cite{BP}.
\begin{lemma}\label{CSS-cvx}
Let $\Omega$ be a convex set different from a ball, with $|\Omega|\leq 1$. Let $\phi_t$, $t\in [0,+\infty]$, be the continuous Steiner symmetrization which maps $\phi_0(\Omega)=\Omega$ into $\phi_\infty(\Omega)=|\Omega|^{1/N}\B$. Then, for every $t$, $\phi_t(\Omega)$ is convex, $|\phi_t(\Omega)|= |\Omega|$, and the functions
$$
t\mapsto \lambda_1(\phi_t(\Omega))\,,\quad t\mapsto T(\phi_t(\Omega))^{-1}
$$
are continuous, with respect to the Hausdorff metric, and decreasing.
\end{lemma}
\begin{remark}\label{rem-CSS}
Composing a continuous Steiner symmetrization with the pair of shape functionals $(\X,\Y)$, we find a continuous path which connects a convex set $\Omega$ to the ball of the same volume. Moreover, in the diagram the path goes downwards in both $x$ and $y$ coordinates. 
\end{remark}

\medskip

The third and last deformation that we recall is the so-called {\it Minkowski sum} of two convex bodies $A$ and $B$:
$$
A\oplus B := \{ a + b \ :\ a \in A,\ b \in B\}.
$$
A classical reference on this subject is \cite{S}.
Given two convex sets $\Omega_0$ and $\Omega_1$ of unit measure, we define the path
\begin{equation}\label{mink}
t\mapsto \frac{t \Omega_1 \oplus (1-t) \Omega_0}{|t\Omega_1 \oplus (1-t) \Omega_0|^{1/N}}.
\end{equation}
Such function deforms in a continuous way (with respect to the Hausdorff metric) $\Omega_0$ into $\Omega_1$, preserving the volume and convexity.
Composing the function above with the pair of shape functionals $(\X,\Y)$, we obtain a continuous curve in $\E$ which connects $(\X(\Omega_0), \Y(\Omega_0))$ and $(\X(\Omega_1), \Y(\Omega_1))$. Such kind of curve will be referred to as \emph{normalized Minkowski curve}. 

\begin{remark}\label{rigid}
Notice that the points $(\X(\Omega_0), \Y(\Omega_0))$ and $(\X(\Omega_1), \Y(\Omega_1))$ are invariant under rigid motion of $\Omega_0$ and $\Omega_1$. But the Minkowski sum isn't. In particular, if we consider in \eqref{mink} the Minkowski sum $t \Omega_1 \oplus (1-t) \Phi(\Omega_0)$, being $\Phi$ a rigid motion, after composing with $(\X,\Y)$, we might find different paths in $\E$, still connecting the same endpoints.
\end{remark}

\subsection{Shape derivatives in dimension 2}

In this paragraph we recall the definition of shape derivatives of order 1 and 2 at $\B$, with respect to smooth deformations which preserve both convexity and volume. In this paragraph we work in dimension $N=2$.

The first order shape derivative of a functional $F$ at $\B$ in direction $V\in C^2(\mathbb R^2;\mathbb R^2)$, if it exists, is defined as
$$
F'(\B;V)\coloneqq \lim_{\e \to 0} \frac{F(\Omega_\e)-F(\B)}{\e},
$$  
where $\Omega_\e \coloneqq (I+\e V)(\B)$. Similarly, taking two vector fields $V,W\in C^2(\mathbb R^2;\mathbb R^2)$ and $\Omega_\e\coloneqq (I+\e V + \frac{\e^2}{2} W)$, the second order shape derivative, if it exists, reads
$$
F''(\B;V,W)\coloneqq \lim_{\e \to 0}2\, \frac{F(\Omega_\e)-F(\B)- \e F'(\B;V)}{\e^2}.
$$
We will focus our attention on a particular class of deformations acting on the class of admissible shapes
\begin{equation}\label{defA}
\mathcal A\coloneqq \{\Omega\subset \mathbb R^N \ : \ \Omega \hbox{ convex},\ |\Omega|= 1\}.
\end{equation}

\begin{definition}\label{admissible} We say that $V,W\in C^2(\mathbb R^2;\mathbb R^2)$ define an {\it admissible deformation} in {$\mathcal A$} if the sets $\Omega_\e\coloneqq (I+ \e V+ \e^2/2 W) (\B)$ {belong to $\mathcal A$}, for every $\e>0$ small enough.
\end{definition}

By \eqref{SV} and \eqref{FK}, the ball is a critical shape for both $T$ and $\lambda_1$ under volume constraint, therefore 
$$
T'(\B;V)=\lambda_1'(\B;V)=0
$$
for every admissible deformation $V$ in {$\mathcal A$}. 

The computation of the second order shape derivatives is more delicate and requires some preliminaries.
We choose to work with support functions (for the definition see, e.g., \cite{SZ}): the support function of $\B$ is constant and equals to the radius $R:=1/\sqrt{\pi}$ of the disk itself; whereas the support function of $\Omega_\e$ is a small perturbation of the constant $R$, of the following form: 
$$
R+ \e \alpha(\theta)+\frac{ \e^2}{2} \beta(\theta),
$$ 
being $\alpha$ and $\beta$ two suitable $2\pi$-periodic functions. 
The relation between $\alpha,\beta$ and $V,W$ is the following: on the boundary $\partial \B$, parametrized with the angular coordinate $\theta \in [0,\pi]$, we have
\begin{equation}\label{supporto}
V(R\cos\theta, R\sin \theta) = \alpha(\theta) n + \dot{\alpha}(\theta) \tau,\quad W(R\cos\theta, R\sin \theta) = \beta(\theta) n + \dot{\beta}(\theta) \tau,
\end{equation}
being $n=(\cos\theta, \sin\theta)$ the unit normal and $\tau=(-\sin\theta, \cos\theta)$ the unit tangent. Here a dot function represents the derivative function.
\\
In order to state the following result, it is convenient to write $\alpha$ and $\beta$ in Fourier series:
\begin{equation}\label{Fourier}
\alpha(\theta)=a_0 + \sum_{m\geq 1}[ a_m\cos(m\theta) + b_m \sin (m\theta)], \  \beta(\theta)= c_0 + \sum_{m\geq 1}[ c_m\cos(m\theta) + d_m \sin (m\theta)].
\end{equation}
The representation with support functions encodes the convexity constraint. As for the area constraint, it results in the following necessary condition (see Lemma \ref{optcond} and Remark \ref{lastremark} in the Appendix):
\begin{equation}\label{lastformula}
a_0=0, \quad c_0 = \frac{\pi}{R}\sum_{m\geq 1} (m^2-1)(a_m^2 + b_m^2).
\end{equation}
We are now in a position to state the following result.

\begin{proposition}\label{prop-ds} Let $V$ and $W$ be admissible deformations in {$\mathcal A$}. Then
\begin{align}
\lambda_1''(\B;V,W) & = 2\pi^2  j_{0,1}^2 \sum_{m\geq 2}\left[ \left( 1 +  j_{0,1}\frac{J_m'(j_{0,1})}{J_m(j_{0,1})} \right) (a_m^2+b_m^2)\right] ,\label{L''}
\\
T''(\B;V,W)& =  - \frac{1}{2} \sum_{m\geq 2} \left[(m-1)(a_m^2 + b_m^2)   \right],\label{T''}
\end{align}
where $\{a_m,b_m\}$ are the Fourier coefficients associated to $V$ as in \eqref{Fourier}-\eqref{supporto}, $J_m$ is the $m$-th Bessel function, and $j_{0,1}$ is the first zero of $J_0$.
\end{proposition}

We postpone the details of the proof to the Section \ref{sec-app}.

\section{The diagram $\E$}\label{sec-E}
This section is devoted to the proof of Theorem \ref{thmE}.

\subsection{Upper and lower bounds}
We recall the definition \eqref{defA} of the class of admissible sets $\mathcal A$ and, for every $x\geq \lambda_1(\B)$, {we introduce} the subfamily
$$
\mathcal A(x)\coloneqq \{\Omega \in \mathcal A\ :\ \lambda_1(\Omega)=x  \}.
$$
Clearly, we have $\mathcal A=\cup_{x\geq \lambda_1(\B)} \mathcal A(x)$. Notice that $\mathcal A(x)$ are all non empty: for a fixed $x$, it is enough to take a parallelepiped $R$ of unit volume and sufficiently small width, so that $\lambda_1(R)>x$; then, taking a continuous Steiner symmetrization of $R$ (see Lemma \ref{CSS-cvx}) we obtain a continuous family of convex sets, whose first Dirichlet eigenvalue runs from $\lambda_1(\B)$ to $\lambda_1(R)$, covering all the interval, including $x$.
According to this notation, the points $(x,y)\in \E$ are of the form $(x,T(\Omega)^{-1})$, for some $\Omega \in \mathcal A(x)$.  

For every $x\geq \lambda_1(\B)$, the diagram is bounded above and below by the following functions:
\begin{equation}\label{Lpm}
L^+(x) \coloneqq \max \left\{T(\Omega)^{-1}\ :\ \Omega \in \mathcal A(x) \right\},\quad L^-(x) \coloneqq \min \left\{ T(\Omega)^{-1}\ :\ \Omega \in \mathcal A(x)\right\}.
\end{equation}
The existence of the maximum and the minimum is a direct consequence of Lemma~\ref{lemma.h}.

\begin{proposition}\label{prop-Lpiu}
The function $L^+$ is upper semicontinuous and continuous from the left. 
\end{proposition}
\begin{proof} 
We first prove the upper semicontinuity. Let $x\geq \lambda_1(\B)$ and let $x_n \to x$ be fixed. We notice that $\sup_n L^+(x_n)$ is bounded, since in view of \eqref{ubE} there holds $L^+(x_n)\leq x_n/C_N$. Therefore, up to a subsequence (not relabeled), we may assume that $ \limsup_n L^+(x_n)=\lim_n L^+(x_n)$. In view of Lemma \ref{lemma.h}, there exists a family of shapes $\Omega_n\in \mathcal A(x_n)$ such that $L^+(x_n)=\Y(\Omega_n)$ and we may find a subsequence $n_k$ and a convex set $\Omega$ satisfying the following properties: the sets $\Omega_{n_k}$ converge to $\Omega$ in the Hausdorff metric as $k\to +\infty$, the limit set $\Omega$ belongs to $\mathcal A(x)$, and $\Y(\Omega)= \lim_k L^+(x_{n_k})$. 
In particular, since $\Omega$ is a competitor for $L^+(x)$ and since by construction $\lim_k L^+(x_{n_k})=\limsup_n L^+(x_n)$, we deduce 
$$L^+(x)\geq Y(\Omega)=\limsup_{n\to\infty} L^+(x_n),$$
namely that $L^+$ is upper semicontinuous.

We now prove the continuity from the left. Let $x\geq \lambda_1(\B)$ and $x_n \to x^-$ be fixed. Thanks to the upper semicontinuity it is enough to prove that 
$$\liminf_{n\to\infty} L^+(x_n)\geq L^+(x).$$
Assume by contradiction that $\liminf_n L^+(x_n) < L^+(x)$. Up to extract a subsequence we may assume that  $\liminf_n L^+(x_n)=\lim_n L^+(x_n)$. Then, for $\epsilon>0$ fixed, we have $L^+(x_n)<L^+(x)-\e$ for $n$ large enough. Let $\Omega\in \mathcal A(x)$ be a set such that $L^+(x)=\Y(\Omega)$. Performing a continuous Steiner symmetrization of $\Omega$ (see Lemma \ref{CSS-cvx} and Remark \ref{rem-CSS}), we construct a (continuous) curve $(x(t),y(t))$, $t\in [0,+\infty]$, contained into the diagram $\E$, with $x(t)=\X(\phi_t(\Omega))$, $y(t)=\Y(\phi_t(\Omega))$, $\phi_t(\Omega)$ convex sets of unit volume, and such that $x(t)\leq x$ and $y(t)\leq L^+(x)$. On one hand, by definition, we have that $L^+(x)-\e > L^+(x_n)=L^+(x(t))\geq \Y(\phi_t(\Omega)) =y(t)$; on the other hand, by the continuity of the continuous Steiner symmetrization, we have $y(t) > L^+(x) -\e$ for all $t$ small enough. This gives the desired contradiction.
 \end{proof}

\begin{proposition}\label{prop-Lmeno}
The function $L^-$ is lower semicontinuous and increasing (then it is continuous from the left).
\end{proposition}
\begin{proof}
The proof follows the lines of Proposition \ref{prop-Lpiu}. The lower semicontinuity follows from the fact that $L^-$ is defined as a minimum; the monotonicity follows by contradiction: if there would exist $x_1<x_2$ with $L^-(x_1)\geq L^-(x_2)$ then the continuous Steiner symmetrization, see Lemma~\ref{CSS-cvx}, starting from $(x_2, L^-(x_2))$ would contradict the minimality of $L^-(x_1)$.
\end{proof}

We denote by $\Gamma^+$ and $\Gamma^-$ the graphs of $L^+$ and $L^-$, respectively:
\begin{equation}\label{bordi}
\Gamma^\pm:= \mathrm{graph}L^\pm = \{(x,L^\pm(x))\ :\ x\geq \lambda_1(\B)\}.
\end{equation}
Note that $\Gamma^+\cup\Gamma^-\subset \partial \E$. If $L^\pm$ were continuous, we would have that their graphs are the upper and lower boundaries of $\E$.

We now focus our attention on the behavior of $L^\pm$ near $x=\lambda_1(\B)$. To this aim, we introduce an auxiliary family of shape functionals: given $\gamma \in \mathbb R$, we set 
\begin{equation}\label{Fgamma}
F_\gamma(\Omega)\coloneqq  \frac{1}{T(\Omega)} - \gamma \lambda_1(\Omega), 
\end{equation}
where $\Omega$ varies in $\mathcal A$.
\begin{definition}\label{def-locminmax}
We say that $\Omega^*\in \mathcal A$ is a {\it local minimizer [resp. maximizer]} for $F_\gamma$ if there exists $\e>0$ such that, for every $\Omega\in \mathcal A$,
\begin{equation}\label{locmin}
|\lambda_1(\Omega)-\lambda_1(\Omega^*)|<\e \quad \Rightarrow \quad  F_\gamma(\Omega)\geq  F_\gamma(\Omega^*) \quad [\hbox{resp. }\  F_\gamma(\Omega)\leq  F_\gamma(\Omega^*)].
\end{equation}
\end{definition}

\begin{proposition}\label{gammameno} Let $x_0\coloneqq \lambda_1(\B)$ and set the maximal and minimal slope at $\mathbb V$ as
\begin{equation}\label{gammapm}
\gamma^+\coloneqq {\sup_{x_n \searrow  x_0} \limsup_{n\to \infty}} \frac{L^+(x_n) - L^+(x_0)}{x_n- x_0},\quad 
\gamma^-\coloneqq {\inf_{x_n \searrow  x_0} \liminf_{n\to \infty} }\frac{L^-(x_n) - L^-(x_0)}{x_n- x_0}.
\end{equation}
Then $\gamma^\pm$ are in $[0,+\infty]$ and admit the following characterization:
\begin{align*}
\gamma^+  =  \inf\{\gamma\ :\ \B\ \hbox{is a local maximizer of}\  F_\gamma\} = \sup\{\gamma\ :\ \B\ \hbox{is not a local maximizer of}\  F_\gamma\},
\\
\gamma^-  = \sup\{\gamma\ :\ \B\ \hbox{is a local minimizer of}\  F_\gamma\} = \inf\{\gamma\ :\ \B\ \hbox{is not a local minimizer of}\  F_\gamma\}.
\end{align*}
\end{proposition}
\begin{proof} We start by studying the function $L^-$. Let $I$ denote the family of parameters $\gamma$ for which the ball is a local minimizer of $F_\gamma$, and by $J$ its complement.
A characterization of the local minimality of $\B$, alternative to \eqref{locmin}, is
\begin{equation}\label{jeudi}
\frac{\Y(\Omega)-\Y(\B)}{\X(\Omega)-\X(\B)}\geq \gamma,
\end{equation}
for every shape $\Omega\in \mathcal A$ such that $0<\X(\Omega) - \X(\B)<\e$ for some $\e>0$ independent of $\Omega$. Such characterization implies that $I$ and its complement $J$ are two intervals. Moreover, $I$ and $J$ are not empty: on one hand, since the ball is a global minimizer for both $X$ and $Y$ (see \eqref{FK} and \eqref{SV}), we infer that every $\gamma\leq 0$ belongs to $I$; on the other hand, taking into account the characterization \eqref{jeudi} and recalling that the diagram is bounded above by the Kohler-Jobin curve $\Gamma_{\mathbb B}$, we infer that every $\gamma > (N+2)/(2T(\B)\lambda_1(\B))$ belongs to $J$. All in all, we infer that 
$$
\sup_I \gamma = \inf_J \gamma <+\infty.
$$

Let $x_n$ be an arbitrary sequence converging to $x_0^+$, let $\gamma$ be an arbitrary element of $I$, and let $\e>0$ be associated to $\gamma$ according to Definition \ref{def-locminmax}. Denote by $\Omega_n$ the shapes in $\mathcal A$ such that $\X(\Omega_n)=x_n$ and $\Y(\Omega_n)=L^-(x_n)$, whose existence in ensured by Lemma \ref{lemma.h}.
By convergence, we infer that, for $n$ large enough, $|x_n-x_0|<\e$. By the characterization \eqref{jeudi} of local minimality, we deduce that
$$
\frac{L^-(x_n)- L^-(x_0)}{x_n-x_0}\geq \gamma.
$$
By the arbitrariness of $\gamma$ in $I$, we conclude that
\begin{equation}\label{gammasup}
\liminf_{n\to \infty} \frac{L^-(x_n)- L^-(x_0)}{x_n-x_0} \geq \sup_I \gamma.
\end{equation}
Let now $\gamma\in J$. Since the ball is not a local minimizer for $F_\gamma$, we may find a sequence of shapes $\Omega_n$ in $\mathcal A$ such that ${\widehat{x}_n}\coloneqq \X(\Omega_n)\to x_0^+$, and ${\widehat{y}_n}\coloneqq \Y(\Omega_n)$ satisfy
$$
\frac{{\widehat{y}_n}-y_0}{{\widehat{x}_n}-x_0} \leq \gamma.
$$ 
By definition, we have $L^-({\widehat{x}_n})\leq {\widehat{y}_n}$, so that
$$
\frac{L^-({\widehat{x}_n})-L^-(x_0)}{{\widehat{x}_n}-x_0} \leq \gamma.
$$
By the arbitrariness of $\gamma \in J$ we get
\begin{equation}\label{gammainf}
\liminf_{n\to \infty} \frac{L^-({\widehat{x}_n})- L^-(x_0)}{{\widehat{x}_n}-x_0} \leq \inf_J \gamma.
\end{equation}
By combining \eqref{gammasup} and \eqref{gammainf} we finally obtain the desired property for $L^-$ at $x_0$.

The statement for $L^+$ can be derived following the same procedure presented for $L^-$ and $\gamma^-$.
\end{proof}

\subsection{Slopes at the vertex $\mathbb V$ in dimension 2}\label{slopes}
In this subsection we find two bounds for the slopes $\gamma^\pm$ at the vertex $\mathbb V$, whose definition is given in \eqref{gammapm}. In this subsection we focus on the planar case $N=2$; a comment for the general dimension is postponed to Remark~\ref{rem-slopes}.

The computation relies on shape derivatives techniques.

\begin{lemma}\label{lem-casi}  Let $F_\gamma$ be the family of shape functionals in \eqref{Fgamma}. Then the following implications hold:
\begin{align*}
 \gamma <  \displaystyle{\frac{32}{j_{0,1}^2(j_{0,1}^2-2)}} \quad & \Rightarrow \quad F''_\gamma(\B;V,W)>0\quad \forall\, V,W\ \hbox{admissible},
\\
 \gamma> \frac{16}{j_{0,1}^2} \quad & \Rightarrow \quad F''_\gamma(\B;V,W)<0\quad \forall\, V,W\ \hbox{admissible}.
\end{align*}
In the intermediate cases, when $32/[j_{0,1}^2(j_{0,1}^2-2)] < \gamma < 16/j_{0,1}^2$, the second order shape derivative does not have constant sign, namely there exist $V_0, W_0$ and $V_1,W_1$ admissible such that
$$
 F''_\gamma(\B;V_0,W_0)<0,\quad  F''_\gamma(\B;V_1,W_1)>0.
$$
\end{lemma}

\begin{proof}

Let $\gamma\in \mathbb R$ and $V,W$ be admissible deformations in {$\mathcal A$}. According to Definition \ref{admissible}, the corresponding small deformations preserve convexity and volume. The volume constraint induces a relation between $V$ and $W$, so that the second order shape derivatives of $T$ and $\lambda_1$ can be written in terms of the sole vector field $V$ (cf. Proposition \ref{prop-ds}). For this reason, in the rest of the proof, $W$ will be omitted. 
Since both $T'$ and $\lambda_1'$ vanish at $\B$, we easily obtain
$$
F_\gamma''(\B;V)= -\frac{T''(\B;V)}{T(\B)^2} - \gamma \lambda_1''(\B;V)\,.
$$
In view of \eqref{L''} and \eqref{T''}, we get
\begin{equation}\label{alter}
F_\gamma''(\B;V) = 2\pi^2 \sum_{m\geq 2} \left(1 + j_{0,1}\frac{J_m'(j_{0,1})}{J_m(j_{0,1})} \right) (r_m-\gamma)  (a_m^2+b_m^2),
\end{equation}
where, for brevity, we have set
$$
r_m\coloneqq \frac{16(m-1)}{j_{0,1}^2\left(1 + j_{0,1}\frac{J_m'(j_{0,1})}{J_m(j_{0,1})}\right)},\quad \hbox{for }m\geq 2.
$$
We claim that
\begin{equation}\label{rmseq}
0<r_2 \leq r_m < 16/{j_{0,1}^2}= \lim_{m\to \infty} r_m.
\end{equation}
These inequalities and the asymptotic behavior of $r_m$ are the consequence of the following estimates, whose statement and proof can be found in \cite[Lemma 11]{Kra1} and \cite[Theorem 1]{Kra2}:
$$
y \frac{J_m'(y)}{J_m(y)}\geq m - \frac{2y^2}{2m+1}\,,\quad y \frac{J_m'(y)}{J_m(y)} \leq \frac{4y^2 - 12 m - 6 + \sqrt{(\mu-4y^2)^3+\mu^2}}{2 [ (2m+1)(2m+5) - 4 y^2]},
$$
valid for $0\leq y < m+1/2$, with $\mu\coloneqq (2m+1)(2m+3)$. Actually, a numerical computation {suggests} that $r_m$ is an increasing sequence, from $r_2$ to $16/j_{0,1}^2$.

Finally, exploiting the positivity of $1+j_{0,1} J_m'(j_{0,1})/J_m(j_{0,1})$ and \eqref{rmseq}, we infer that if $\gamma$ is below $r_2$ or above $16/j_{0,1}^2$, the derivative $F''_\gamma$ has constant sign, positive and negative, respectively, for every admissible deformation. On the other hand, if $\gamma$ is strictly between the two values, there exist suitable choices of $a_m$ and $b_m$ which make the derivative positive or negative. Exploiting the well known properties of the Bessel functions and their derivatives, we obtain the expression $r_2=32/[j_{0,1}^2(j_{0,1}^2-2)]$. This concludes the proof.
\end{proof}

The thresholds appearing in Lemma \ref{lem-casi} give the values of $\gamma^\pm$, as we state in the following.

\begin{proposition}\label{prop-valorislopes} 
In dimension $N=2$ the minimal and maximal slopes introduced in Proposition \ref{gammameno} satisfy
\begin{equation}\label{defgammameno}
\gamma^+= \frac{16}{j_{0,1}^2}  \quad \hbox{and} \quad  \gamma^- {\leq} \frac{32}{j_{0,1}^2(j_{0,1}^2-2)}.
\end{equation}
\end{proposition}
\begin{proof}
In Proposition \ref{gammameno} we have characterized $\gamma^-$ as the infimum of the $\gamma$s for which the ball (here, the disk) is not a local minimizer. In view of the statement of Lemma \ref{lem-casi}, in particular the conclusions concerning the intermediate cases  $32/[j_{0,1}^2(j_{0,1}^2-2)] < \gamma < 16/j_{0,1}^2$, we clearly have $\gamma^- \leq 32/[j_{0,1}^2(j_{0,1}^2-2)]$ and $\gamma^+\geq 16/j_{0,1}^2$. The maximal slope $\gamma^+$ is bounded above by the slope of the Kohler-Jobin curve $\Gamma_{\B}$ at $\mathbb V$, which in dimension 2 equals $16/j_{0,1}^2$. This concludes the proof.
\end{proof}

We conclude the paragraph with two remarks.

\begin{remark}\label{DE}
As already noticed in the proof of Proposition \ref{prop-valorislopes}, $\gamma^+$ agrees with the slope of the Kohler-Jobin curve $\Gamma_{\mathbb B}$ at $\mathbb V$.  This fact is surprising, since the diagrams $\E$ and $\D$ touch the vertex $\mathbb V$ a priori in two different ways. {Notice that $L^+$ is differentiable at $\lambda_1(\B)$, indeed, for every $\gamma<\gamma^+$ and for every $x_n\to x_0:=\lambda_1(\B)$, in view of Lemma \ref{lem-casi} and Proposition \ref{prop-valorislopes}, we may find a family of shapes $\Omega_{\e_n}$ (cf. Definition \ref{admissible}) such that $X(\Omega_{\e_n})=x_n$ and $F_\gamma(\Omega_{\e_n})\geq F_\gamma (\B)$. Therefore, using that $L^+$ is defined as a maximum and exploiting Proposition \ref{prop-valorislopes}, we get
$$
\gamma\leq \liminf_{n\to \infty} \frac{Y(\Omega_{\e_n})-Y(\B)}{X(\Omega_{\e_n})-X(\B)} \leq \liminf_{n\to \infty} \frac{L^+(x_n)-L^+(x_0)}{x_n-x_0} \leq \gamma^+.
$$
By the arbitrariness of $x_n$ and that of $\gamma$, we obtain the differentiability of $L^+$ at $x_0$ with derivative $\gamma^+$.}
\end{remark}

\begin{remark}\label{rem-slopes} The computation done in the planar case could be, in principle, repeated in higher dimension: the description with support functions still applies, but has to be done in a specific way according to $N$ (see also \cite{Bogo}).
\end{remark}

\subsection{A related shape optimization problem}

The relation between the minimization of $F_\gamma$ (introduced in \eqref{Fgamma}) and the boundary of the diagram goes beyond the \emph{local} analysis presented in the previous paragraph, performed near the ball for $F_\gamma$ and near $\mathbb V$ for $L^-$. Actually, if $\Omega^*$ is a \emph{global} minimizer of $F_\gamma$ for some $\gamma$, it is immediate to check that it also minimizes $1/T$ keeping $\lambda_1$ fixed. More precisely, the line $y=\gamma x + F_\gamma (\Omega^*)$ is tangent to $\Gamma^-$ at $(\lambda_1(\Omega^*), T(\Omega^*)^{-1})$ and lies below the diagram $\E$. Also the non existence of minimizers gives some information: if $\inf F_\gamma=q \in \mathbb R$ but the infimum is not attained, this means that the line $y=\gamma x + q$ lies below the diagram and is an asymptote for $\Gamma^-$; if instead the infimum is $-\infty$, it means that for every $q\in \mathbb R$ there exists a point of the diagram which lies below the line $y=\gamma x + q$ (for topological reasons, it means that each of these lines crosses $\Gamma^-$).

In this paragraph we prove the following:
\begin{proposition} There exist two real numbers $0\leq\gamma_0\leq \gamma_1<+\infty$ such that the following facts hold true: in the class $\mathcal A$,
\begin{itemize}
\item[(i)] for every $\gamma \leq \gamma_0$ the ball minimizes $F_\gamma$,
\item[(ii)] for every $\gamma\in (\gamma_0,\gamma_1)$ a minimizer for $F_\gamma$ exists and is not a ball,
\item[(iii)] for every $\gamma >\gamma_1$ the functional $F_\gamma$ does not have a minimizer, 
\end{itemize}
\end{proposition}
\begin{proof}
First we show that the values of $\gamma$ for which $F_\gamma$ admits a minimizer are in an interval of the form $(-\infty, \gamma_1)$, namely if $F_\gamma$ admits a minimizer for some $\gamma$, then all the functionals $F_{\gamma'}$ with $\gamma'<\gamma$ admit a minimizer: by the relation between $\gamma'$ and $\gamma$, it is immediate to check that $F_{\gamma'}$ is bounded below, so that its infimum $\ell'$ is finite; a minimizing sequence $\Omega_n$ satisfies, for $n$ large enough,
$$
\ell' + 1 \geq  F_{\gamma'}(\Omega_n) = F_{\gamma}(\Omega_n) + (\gamma-\gamma')\X(\Omega_n) \geq \ell + (\gamma-\gamma')\X(\Omega_n)
$$
with $\ell:=\inf F_\gamma\in \mathbb R$. In particular the $\X$ coordinate functional is bounded along the sequence $\Omega_n$. This condition, together with the assumption $|\Omega_n|=1$ for every $n$, provides the compactness which ensures the existence of a minimizer (see Lemma \ref{lemma.h}).

The value of $\gamma_1$ is not known, nevertheless, we claim that it is a positive number. More precisely, we prove that $1\leq \gamma_1 \leq 1/C_N$, being $0< C_N < 1 $ the positive dimensional constant appearing in \eqref{ubE}. In view of the P\'olya inequality \eqref{P}, we have 
$$
\Y-\gamma \X \geq (1-\gamma) \X \geq (1-\gamma)\X(\B),
$$
in particular, if $\gamma<1$, we infer that $F_\gamma$ is bounded below and along a minimizing sequence the functional $\X$ is bounded. As above, these two facts imply existence of minimizers for every such $\gamma$. On the other hand, using the estimate \eqref{ubE}, stating that $Y\leq \X / C_N  $, we have
$$
\Y-\gamma \X \leq \left(1/C_N-\gamma\right) \X
$$ 
which leads to $\inf F_\gamma = -\infty$ whenever $\gamma>1/C_N$. 

Let us investigate the role of the ball. In view of the Faber-Krahn inequality \eqref{FK}, it is immediate to check that the $\gamma$s for which the ball is optimal are in an interval of the form $(-\infty, \gamma_0)$, with $\gamma_0\leq \gamma_1$: as before, if $\B$ is a minimizer of $F_\gamma$, then it is a minimizer also for $\gamma'<\gamma$, since
$$
F_{\gamma'}(\Omega)= F_\gamma(\Omega) + (\gamma-\gamma')\X(\Omega) \geq F_\gamma(\B) + (\gamma-\gamma')\X(\B) = F_{\gamma'}(\B).
$$
This proves the proposition.
\end{proof}

The precise values of $\gamma_0$ and $\gamma_1$ are unknown and a priori could coincide. Note that one could also address the maximization problem of the family $F_\gamma$, and arrive to similar conclusions.
 
\subsection{Topology of the diagram} 

In this paragraph we investigate the topology of $\E$. In Proposition \ref{chiusoecpa} we show that the diagram is closed and connected by arcs. Then, in Proposition \ref{noholes1} we exclude the presence of unbounded holes.  

\begin{proposition}\label{chiusoecpa}
The diagram $\E$ is closed and connected by arcs.
\end{proposition}
\begin{proof} The closure is a consequence of Lemma \ref{lemma.h}, in which continuity and compactness of sequences of bounded convex sets are stated. As for the connectedness, it is a consequence of the continuous Steiner symmetrization (see \S \ref{sec-CSS}): any point $(x,y)\in \E$ can be connected to the vertex $\mathbb V$ following the continuous path obtained composing $(\X, \Y)$ with a continuous Steiner symmetrization of a set $\Omega$ such that $\lambda_1(\Omega)=x$, $T(\Omega)^{-1}=y$.
\end{proof}

Let us show that no unbounded hole can occur in the diagram.

\begin{proposition}\label{noholes1}
The boundary $\partial \E$ has only one unbounded connected component.
 \end{proposition}
 \begin{proof}
 Some ideas of the following proof are inspired by {\cite{FL}}, in which the authors study the Blaschke-Santal\'o diagram of the pair (perimeter, $\lambda_1$), under volume constraint.  

In order to prove the statement, it is enough to exclude the presence of unbounded holes into the diagram. Assume by contradiction that there exists an open set $A$ such that:
 \begin{itemize}
 \item[i)] $A$ is simply connected and $A\cap \E=\emptyset$,
 \item[ii)] for every point $(x,y)\in A$, there holds $L^-(x) <y < L^+(x)$,
 \item[iii)] $A$ is unbounded, namely it intersects every half plane $\{x\geq h\}$, for every $h>\lambda_1(\B)$.
 \end{itemize}

Let $x_1>\lambda_1(\B)$ and let $\Phi$ be a rigid motion, which will be suitably chosen later.
For two optimal sets $\Omega_1$ and $\Omega_0$ of ${L}^+(x_1)$ and $L^-(x_1)$, respectively, 
denote by $\Omega_t$ the Minkowski sum $t \Omega_1 \oplus (1-t)\Phi(\Omega_0)$, $t\in [0,1]$, and set the normalized set $\widetilde\Omega_t:= |\Omega_t|^{-1/N}\Omega_t$. As already noticed in Remark \ref{rigid}, the range of the function $t\mapsto (\X(\widetilde\Omega_t), \Y(\widetilde\Omega_t))$ is a continuous curve in $\E$, connecting $(\X(\Omega_0), \Y(\Omega_0))$ to $(\X(\Omega_1), \Y(\Omega_1))$.

We look for a lower bound on the abscissa of the points of such a curve: exploiting the $-2$ homogeneity of $\lambda_1$ we clearly have
\begin{equation}\label{sc1}
\X(\widetilde\Omega_t)=\lambda_1(\widetilde{\Omega}_t) = |\Omega_t|^{2/N} \lambda_1(\Omega_t).
\end{equation}
In view of the Brunn-Minkowski inequality (see \cite{S}), we immediately get $|\Omega_t|\geq 1$. As for $\lambda_1(\Omega_t)$, by the \emph{Hersh-Protter inequality} \cite{Hersh, P},
$$
\lambda_1(\Omega_t)\geq \frac{\pi^2}{4\rho(\Omega_t)^2},
$$
where $\rho(\cdot)$ is the inradius. In general, the inradius of a Minkowski sum is greater than or equal to the sum of the inradii of the addenda; however, there exists a rigid motion $\Phi$ which gives the equality: 
$$
\rho(\Omega_t) = t \rho(\Omega_1) + (1-t)\rho(\Phi(\Omega_0))= t \rho(\Omega_1)+(1-t)\rho(\Omega_0).
$$
These last two facts imply that \eqref{sc1} can be further bounded from below as follows:
\begin{equation}\label{sc2}
\X(\widetilde\Omega_t)\geq \frac{{\pi^2}}{{4}[t \rho(\Omega_1) + (1-t) \rho(\Omega_0)]^2} \geq  \frac{{\pi^2}}{{4}\max \left( \rho(\Omega_0)^2; \rho(\Omega_1)^2 \right)}.
\end{equation}
Notice that if we consider $x_1\to +\infty$, then both the inradii in the right-hand side will go to zero, so that $\X(\widetilde\Omega_t)$ will diverge to $+\infty$. 

Therefore, by taking $x_1$ large enough, the path $t\mapsto (\X(\widetilde{\Omega}_t), \Y(\widetilde{\Omega}_t) )$ cuts the set $A$, in contradiction with (i)-(iii) above.
 \end{proof}

Notice that if $L^\pm$ were continuous, in view of the last proposition, we would have that the unbounded connected component of $\partial \E$ coincides with $\Gamma^+\cup \Gamma^-$.

\section{The diagram $\D$}\label{sec-D}
This section is devoted to the proof of Theorem \ref{thmD}, in which the analysis of the diagram $\D$ with volume constraint is taken in the weak form. The study of $\D$ is closely related to that of $\E$: on one hand, the setting is less rigid  (e.g., contractions of admissible sets are now admissible) and many properties of $\E$ are easily inherited by $\D$; on the other hand, some properties of $\D$ are a posteriori verified by sets of unit volume, allowing us surprisingly to deduce some unnoticed properties of $\E$. 

\subsection{Upper and lower boundaries}

As already done for $\E$, it is natural to introduce two functions which bound from above and below the diagram. Here, in accordance with \eqref{Lpm}, we define
\begin{align}
\widehat{L}^+(x) & \coloneqq \max \left\{ \frac{1}{T(\Omega)}\ :\ \Omega \hbox{ convex, } |\Omega|\leq 1  \right\},\label{Lphat}
\\
\widehat{L}^-(x) & \coloneqq \min \left\{ \frac{1}{T(\Omega)}\ :\ \Omega \hbox{ convex, } |\Omega|\leq 1\right\},\label{Lmhat}
\end{align}
for $x\geq \lambda_1(\B)$.
The existence of the maximum and of the minimum is a direct consequence of Lemma~\ref{lemma.h}.
As already pointed out in the introduction, the unique optimal set associated to $\widehat{L}^+(x)$ is the ball $r\B$, with $r=\sqrt{\lambda_1(\B)/x}$. Moreover,  
\begin{equation}\label{formulaL+}
\widehat{L}^+(x)= \frac{1}{T(\B)\lambda_1(\B)^{(N+2)/2}}\,x^{(N+2)/2},
\end{equation}
and its graph is the Kohler-Jobin curve:
$$
\Gamma_{\mathbb B} = \{(x,\widehat{L}^+(x))\ :\ x\geq \lambda_1(\B)\}.
$$
As a byproduct, we infer that $\widehat{L}^+$ is a continuous curve, increasing, with slope at $\lambda_1(\B)$ equal to $(N+2)/[2 T(\B)\lambda_1(\B)]$ ($=\gamma^+$ in dimension $N=2$).

\begin{remark}\label{rem41} The knowledge of the optimal sets on $\Gamma_{\mathbb B}$ provides a lower bound on the volume of the convex sets in $\D$. Let $\Omega$ be a convex set of volume at most 1, associated to the point $(x,y)$. In view of Lemma~\ref{CSS-cvx}, the ball $B_\Omega$ of volume $|\Omega|$ is necessarily located in the lower left part of $(x,y)$, namely at some $(x_1,L^+(x_1))$ with $x_1\leq x$ and $L^+(x_1)\leq y$. Since $T(B_\Omega)|B_\Omega|^{-(N+2)/N}=T(\mathbb B)$ we infer that
\[
|\Omega|=|B_\Omega|=\left(\frac{T(B_\Omega)}{T(\B)} \right) ^{\frac{N}{N+2}}\geq\left(\frac{\Y(\B)}{y} \right) ^{\frac{N}{N+2}}.
\]
Moreover, from $\lambda_1(B_\Omega)|B_\Omega|=\lambda_1(\mathbb B)$ we also have
\[
|\Omega|=|B_\Omega|=\frac{\la_1(\B)}{\la_1(B_\Omega)}\geq \frac{\X(\B)}{x} 
\]
which combined with the previous one yields the lower bound on the measure of $\Omega$, i.e.,
\[
|\Omega|\geq \max\left\{\frac{\X(\B)}{x} ,\left(\frac{\Y(\B)}{y} \right) ^{\frac{N}{N+2}}\right\}.
\]

In particular, as one may expect, sets associated to points near the vertex $\mathbb V$ have almost unit volume.
Notice that, if the value of $x$ and $y$ are not explicitly known, but only the upper bounds $x\leq x_2$ and $y\leq y_2$ are available, one may say
\[
|\Omega| \geq \max\left\{\frac{\X(\B)}{x_2}, \left(\frac{\Y(\B)}{y_2} \right) ^{\frac{N}{N+2}}\right\}.
\]
\end{remark} 

The properties of $\widehat{L}^-$ are less evident and deserve a deeper analysis. In accordance with \eqref{bordi}, we denote by $\widehat{\Gamma}^-$ its graph:
$$
\widehat{\Gamma}^-:=\{(x,\widehat{L}^-(x)\ :\ x\geq \lambda_1(\B)\}.
$$

\begin{proposition}\label{prop-L2} The function $\widehat{L}^-$ is continuous and increasing.
\end{proposition}
\begin{proof} The proof is divided into four steps: the lower semicontinuity, the continuity from the left, the continuity from the right, and the monotonicity of $\widehat{L}^-$. The first property follows by the definition of $\widehat{L}^-$ together with the compactness Lemma \ref{lemma.h}; the continuity from the left and the monotonicity are consequence of the monotonicity of the continuous Steiner symmetrization (see Remark \ref{rem-CSS}). We omit here the complete proof of these steps, since it would retrace that of Propositions \ref{prop-Lpiu} and  \ref{prop-Lmeno}.

The new part of this proof is the continuity from the right, that we detail here. Assume by contradiction that for some $x$ and $x_n \to x^+$ there holds $\widehat{L}^-(x) < \lim_n \widehat{L}^-(x_n)$. Let $\Omega$ be optimal for $\widehat{L}^-$ and $\Omega_t=(1-t)\Omega$, $t\in [0,1)$ be the continuous family of its contractions. In the diagram, these deformations correspond to the curve $\eta(t)= (x(t), y(t))$, with $x(t)=\lambda_1(\Omega_t)$ and $y(t)=T(\Omega_t)^{-1}$. Such curve starts at $t=0$ from $(x,L^-(x))$ and its vertical component is a continuous increasing function of the horizontal component (see Remark \ref{rem-KJ}). Since $x(t)$ runs from $x$ to $+\infty$, we may find $t_n$ such that $x(t_n)=x_n$, for $n\in \mathbb N$. By definition we have $\widehat{L}^-(x_n)\leq 1/T(\Omega_{t_n})$. Passing to the limit as $n\to \infty$  in both sides, we get $\lim_{n} L^-(x_n)\leq L^-(x)$, which is absurd.
 \end{proof}

\begin{proposition}\label{prop43}
The two functions $\widehat{L}^+$ and $\widehat{L}^-$ have the same value only at the point $\lambda_1(\B)$, and this value is equal to $1/T(\B)$.
\end{proposition}
\begin{proof}
The coincidence of $\widehat{L}^+$ and $\widehat{L}^-$ at $x=\lambda_1(\mathbb B)$ is trivial, since $\widehat{L}^-(\lambda_1(\B)) \leq 1/T(\B)$ just by definition of $\widehat{L}^-$ while $\widehat{L}^-(\lambda_1(\B)) \geq 1/T(\B)$ thanks to \eqref{SV}. Then, by \eqref{formulaL+} we deduce that $\widehat{L}^+(x(\B))=\widehat{L}^-(x(\B))$.

Let now $x> \lambda_1(\B)$ be fixed. If we find a convex set $\Omega$ such that $|\Omega|\leq 1$, $\lambda_1(\Omega)=x$, and $1/T(\Omega) < \widehat{L}^+(x)$, we are done, since this would imply $\widehat{L}^-(x)\leq 1/T(\Omega)<\widehat{L}^+(x)$. In view of \eqref{formulaL+}, $x= \lambda_1(r\B)$ and $\widehat{L}^+(x)=1/T(r\B)$, with $r=\sqrt{\lambda_1(\B)/x}<1$. Let $\Omega$ be a convex set with the same volume of $r\B$, namely $|\Omega|=|r\B|=r^{N}$, and such that
\begin{equation}\label{stime}
x< \lambda_1(\Omega)<\frac{x}{r^N}.
\end{equation}
The former inequality is always true in view of \eqref{FK} taking $\Omega$ different from a ball, while the latter is easily satisfied if, e.g., $\Omega$ is chosen close enough to the ball $r\B$.
Dilating $\Omega$ of a factor $t>1$, the first Dirichlet eigenvalue decreases, in particular, by choosing $t\coloneqq \sqrt{\lambda_1(\Omega)/x}$, we get $\lambda_1(t\Omega) = x$. Moreover, thanks to the second inequality in \eqref{stime}, we have $|t\Omega|= t^N |\Omega| <r^{N/2}<1$, so that $t\Omega$ has volume less than or equal to 1 and $\lambda_1(t\Omega)=x$. By a direct computation, we get
$$
\frac{1}{T(t\Omega)} =\frac{1}{t^{(N+2)}T(\Omega)} = \left(\frac{\lambda_1(r\B)}{ \lambda_1(\Omega)  T(\Omega)^{2/(N+2)}} \right)^{(N+2)/{2}} < \frac{1}{T(r\B)}=\widehat{L}^+(x),
$$
where the last inequality follows from the Kohler-Jobin estimate \eqref{KJ}. This concludes the proof.
\end{proof}

Further properties of $\widehat{L}^-$ are given in \S \ref{back}.

\subsection{Topology of the diagram}

In the previous paragraph we have shown that the diagram $\D$ is enclosed between two increasing curves, both starting from $\mathbb V$, diverging to $+\infty$ as $x\to +\infty$, and with no other intersection than $\mathbb V$. They are defined as the graphs of $\widehat{L}^+$ and $\widehat{L}^-$, denoted by $\Gamma_{\mathbb B}$ and $\widehat{\Gamma}^-$, respectively.

\begin{proposition}\label{prop-Dcvx}
The set $\D$ is the region between the curves $\Gamma_{\mathbb B}$ and $\widehat{\Gamma}^-$. In particular, it is closed, simply connected, and convex in the $x$ and $y$ directions.
\end{proposition}
\begin{proof} The closure of $\D$ is a consequence of Lemma~\ref{lemma.h}. As for the simple connectedness, it is enough to show that $\D$ coincides with the region between $\Gamma_{\mathbb B}$ and $\widehat\Gamma^-$. Let $(x_1,y_1)$ be a point lying between the two curves, namely such that
$$
x_1>\lambda_1(\B)\quad \hbox{and}\quad     \widehat L^-(x_1) < y_1 <  \widehat L^+(x_1),
$$
where $\widehat L^\pm$ are the functions defined in \eqref{Lphat} and \eqref{Lmhat}. The equality cases, corresponding to the upper and lower curves, have already been treated in the previous paragraph. Since $\widehat \Gamma^{-}\cup \Gamma_{\mathbb B}$ disconnects the plane into two parts, the former containing $(x_1,y_1)$, the latter containing the origin, we infer that any curve connecting these two points must intersect $\Gamma_{\mathbb B}$ or $\widehat \Gamma^{-}$. In particular, the curve $\Gamma_1\coloneqq\{y= c_1\,x^{(N+2)/N}\}$, with $c_1\coloneqq y_1/x_1^{(N+2)/N}$, which passes through the origin and $(x_1,y_1)$, has to intersect $\widehat\Gamma^{-}$ at some $(x_2,y_2)$, with $\lambda_1(\B)< x_2 < x_1$ and $y_2=c_1\, x_2^{(N+2)/N}$. Since $\widehat\Gamma^-$ is contained into the diagram, we infer that also $(x_2,y_2)\in \D$. As already noticed in \S \ref{sec-CSS}, the whole arc $\Gamma_2\coloneqq \{y=c_2 \, x^{(N+2)/N}\ :\ x\geq x_2\}$, with $c_2\coloneqq y_2/x_2^{(N+2)/N}$, is contained into $\D$. Since $(x_2,y_2)\in \Gamma_1\cap \Gamma_2$, it is immediate to check that $c_1=c_2$, namely $\Gamma_1$ and $\Gamma_2$ are actually the same (more precisely, $\Gamma_2$ is a portion of $\Gamma_1$). In particular, the point $(x_1,y_1)\in \D$. The procedure is also described in Fig. \ref{figura2}. This gives the simple connectedness. Actually, we have proved a stronger fact: for every $x_1$ as above, the whole segment $\{x_1\}\times[\widehat{L}^-(x_1),\widehat{L}^+(x_1)]$ is contained into the diagram, namely $\D$ is vertically convex; moreover, the same reasoning applies in the horizontal direction, concluding the proof.
\end{proof}

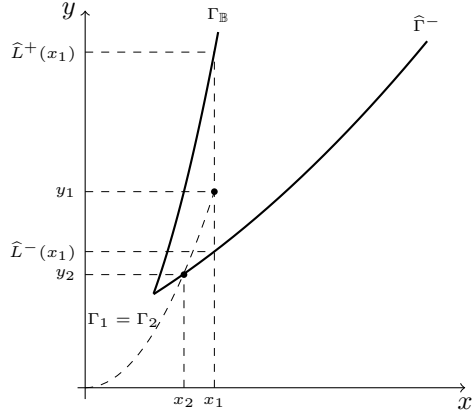
\begin{figure}[t]                                         
\begin{tikzpicture}[domain=0:10, scale=0.5]   

  \draw[->] (-0.2,0) -- (10,0) node[below] {$x$};
  \draw[->] (0,-0.3) -- (0,10) node[left] {$y$};
  \draw[thick,domain=1.8:9]  plot (\x,{(0.2*\x)^2 + 0.5*(\x)+1.45}) node[above] {\tiny $\widehat\Gamma^-$};
  \draw[thick,domain=1.8:3.5]  plot (\x,{(0.77*(\x)^2)}) node[above] {\tiny $\Gamma_{\mathbb B}$};
  
    \fill[black] (3.4,5.2) circle (2.5pt);
    \draw[-, dashed, very thin, domain=0:3.4]  plot (\x,{0.45*(\x)^2}) ;
    \node at (1,1.8) {\tiny $\Gamma_1=\Gamma_2$};
    \draw[-, dashed, very thin] (3.4,0) node[below] {\tiny $x_1$} -- (3.4,8.9) ;
    \draw[-, dashed, very thin] (0,5.2) node[left] {\tiny $y_1$} -- (3.4,5.2) ;
    \draw[-, dashed, very thin] (0,8.9) node[left] {\tiny $\widehat L^+(x_1)$} -- (3.4,8.9) ;
    \draw[-, dashed, very thin] (0,3.61) node[left] {\tiny $\widehat L^-(x_1)$} -- (3.4,3.61) ;
    \draw[-, dashed, very thin] (2.6,0) node[below] {\tiny $x_2$} -- (2.6,3) ;
    \draw[-, dashed, very thin] (0,3) node[left] {\tiny $y_2$} -- (2.6,3) ;
    \fill[black] (2.6,3) circle (2.5pt);
\end{tikzpicture}
\caption{{The construction of $\Gamma_1$ and $\Gamma_2$ in the proof of  Proposition \ref{prop-Dcvx}.}}\label{figura2}
\end{figure}

\subsection{Back to the upper and lower boundaries}\label{back}

In the next propositions we shed some light on the relationship between $\widehat L^-$ and $L^-$.

\begin{proposition}\label{prop-L3} The two functions $\widehat L^-$ and $L^-$ coincide on a closed unbounded set. Moreover, on {the} complement $\widehat L^-$ is a Kohler-Jobin type curve.
\end{proposition}
\begin{proof} 

Let $\Omega_1$ be a convex set with volume $s<1$ such that $(x_1,y_1):=(\X(\Omega_1),\Y(\Omega_1))\in \widehat\Gamma^-$. 

\emph{Step 1:} we claim that all the admissible dilations of $\Omega_1$ are associated to points of $\widehat\Gamma^-$, namely, according to the notation introduced in \S \ref{sec-CSS}, we claim that $\Gamma_{\Omega_1}(1,s^{-1/N})\subset \widehat\Gamma^-$. 
If not, there would exist $t^* \in  (1,s^{-1/N}]$ and some convex set $\Omega^*$ with volume less than or equal to 1, such that $\X(\Omega^*)=\X(t^*\Omega_1)$ and $\Y(\Omega^*)<\Y(t^*\Omega_1)$. The relationship between the coordinates of $\Omega^*$ and $t^*\Omega_1$ implies that the curve $\Gamma_{\Omega^*}(0,1)$, 
which is included in the diagram, lies (strictly) below $\Gamma_{\Omega_1}(1,s^{-1/N})$. In particular, we find a point of the diagram which has the same $x$-coordinate of $\Omega_1$, but strictly less $y$-coordinate, in contradiction with the optimality of $\Omega_1$.

\emph{Step 2:} we claim that not all the contractions of $\Omega_1$ are associated to a point of $\widehat{\Gamma}^-$. Indeed, if not, we would obtain that, for $x$ large enough, the set $\widehat\Gamma^-$ coincides with the curve $\Gamma_{\Omega_1}$. In particular, $\widehat\Gamma^-$ is superlinear, contradicting the bound \eqref{ubE}.

\emph{Step 3:} Let $I:=\{\widehat L^-=L^-\}$. The closedness of $I$ is a consequence of Lemma \ref{lemma.h}. As for the unboundedness, it follows arguing by contradiction and using Step 2. Finally, combining Step 1 and 2, we infer that $\widehat L^-$ is a Kohler-Jobin type curve on each connected component of $I^c$.
\end{proof}

Even if $L^-$ and $\widehat L^-$ a priori do not coincide, they share the same minimal slope, in the sense of \eqref{gammapm}, as we show in the following.
\begin{corollary}\label{corollario}
The boundaries of $\D$ and $\E$ have the same minimal slope $\gamma^-$ at $\mathbb V$.
\end{corollary}
\begin{proof}
Recalling the definition \eqref{gammapm} of $\gamma^-$ and exploiting the inclusion $\E\subset \D$, we infer that the minimal slope of $\widehat L^-$ at $\lambda_1(\B)$ is less than or equal to $\gamma^-$. Assume by contradiction that it is strictly less than $\gamma^-$. Therefore we may find a sequence of sets $\Omega_n$ converging to the ball, with volume $1-\e_n$ for some $\e_n \searrow 0$, and such that $(Y(\Omega_n)-Y(\B))/(X(\Omega_n)-X(\B))\to \gamma<\gamma^-$. Denoting by $\widetilde\Omega_n$ the sequence of normalized sets, by the homogeneity of $X$ and $Y$, we get
\begin{align*}
& \frac{Y(\Omega_n)-Y(\B)}{X(\Omega_n)-X(\B)} = \frac{Y((1-\e_n)^{1/N}\widetilde\Omega_n)-Y(\B)}{X((1-\e_n)^{1/N}\widetilde\Omega_n)-X(\B)} = \frac{Y(\widetilde\Omega_n) - Y(\B) + (N+2)\e_n Y(\widetilde\Omega_n) /N + o(\e_n)}{X(\widetilde\Omega_n) - X(\B) + 2\e_n X(\widetilde\Omega_n)/N + o(\e_n)}
\\
& \geq \min \left( \frac{Y(\widetilde\Omega_n) - Y(\B) }{X(\widetilde\Omega_n) - X(\B)}, \frac{(N+2)Y(\widetilde\Omega_n) + o(1)}{2 X(\widetilde\Omega_n) + o(1)}\right).
\end{align*}
Finally, passing to the limit in $n$, the left-hand side converges to $\gamma$; whereas the right-hand side is bounded from below by $\min(\gamma^-,(N+2)/[2\lambda_1(\B) T(\B)])=\gamma^-$, since $(N+2)/[2\lambda_1(\B) T(\B)]$ is the slope of the Kohler-Jobin curve $\Gamma_\B$ at $\mathbb V$ (see Remark \ref{rem-KJ}). This gives a contradiction concluding the proof.
\end{proof}

\section{Open problems}\label{sec-open}
The first natural questions which arise from our results concern the boundary of the diagram $\E$: are $\Gamma^\pm$ continuous curves? Is the bound in \eqref{eq.slopes} for $\gamma^-$ in dimension 2 optimal? For the latter, we expect that, near $\mathbb V$, the elements of $\Gamma^-$, in view of their optimality for $L^-$, converge in some {\it stronger} sense to the ball, supporting an affirmative answer. In this section we present three other open problems, that we believe to be interesting research lines.

\subsection{Optimal sets on $\Gamma^\pm$}

In order to characterize the optimal shapes on the upper and lower boundaries of $\E$, a natural idea is to write optimality conditions, enclosing the constraints into the functional, via Lagrange multipliers. If $\Omega$ is a critical shape (maximizer or minimizer) for $1/T$ under volume constraint and prescribed $\lambda_1$, say equal to $x$, there exists a Lagrange multiplier $\mu\in \mathbb R$ such that
$$
\left\{
\begin{array}{lll}
& \frac{\de}{\de V} \left(\frac{|\Omega|^2}{T(\Omega)} \right) - \mu  \frac{\de}{\de V} \left(|\Omega|\lambda_1(\Omega) - x \right) = 0
\\
& |\Omega|\lambda_1(\Omega) = x
\end{array}
\right.
$$
for every deformation $V$ which preserves convexity. Here we have denoted, for brevity, the shape derivative in direction $V$ by $\mathrm{d} / \mathrm{d}V$. In case $\Omega$ is smooth and strictly convex on some part $\gamma\subset \partial \Omega$, the deformations fields $V$ can be taken with arbitrary sign on $\gamma$. In this case, taking without loss of generality $|\Omega|=1$ and developing the computations, we get
\begin{equation}\label{Serrin1}
\left\{
\begin{array}{lll}
& |\nabla w_\Omega|^2 - \mu T(\Omega)^2 |\nabla \varphi_\Omega|^2 = 2 T(\Omega) - \mu T(\Omega )^2  \lambda_1(\Omega)  \quad \hbox{on }\gamma,
\\
& \lambda_1(\Omega) = x.
\end{array}
\right.
\end{equation}
The first optimality condition can be rephrased as follows: 
\begin{equation}\label{c-ott}
|\partial_n w_\Omega|^2 - \alpha |\partial_n \varphi_\Omega|^2 = \beta \quad  \hbox{ on }\gamma,
\end{equation}
for some $\alpha, \beta\in \mathbb R$. In other words, the torsion function and the first Dirichlet eigenfunction solve two overdetermined problems, in which the extra-condition involves both $w_\Omega$ and $\varphi_\Omega$. The natural question is: which (if it exists) smooth strictly convex domains satisfy \eqref{c-ott}? For which values $\alpha, \beta$ the sole solution is the ball? A positive answer would imply that the optimal sets different from the ball are either non smooth, or not strictly convex.
\begin{center}
{\bf Open problem 1:} characterize the optimal shapes on $\Gamma^\pm$.
\end{center}
In this respect, we have the following conjecture: the optimal sets on $\Gamma^+$ are polygons, whereas those on $\Gamma^-$ are $C^{1,1}$.

\subsection{Topology of $\E$}
In Theorem \ref{thmE} we have shown that the diagram $\E$ is connected and that the eventual holes are bounded. It would be interesting to exclude, or confirm, the presence of holes; in other words
\begin{center}
{\bf Open problem 2 :} prove or disprove the simple connectedness of $\E$.
\end{center}
Here we provide a nice tool, based on a topological argument, which supports the conjecture of simple connectedness.

To this aim, we need to introduce some notation. Given two convex sets $\Omega_1,\Omega_2$ of $\mathbb R^N$ of unit measure, we define a loop passing through the vertex as follows: first, performing an ``inverse'' continuous Steiner symmetrization, we may pass from $\B$ to $\Omega_1$; then, by applying a normalized Minkowski sum (see \eqref{mink}) we may deform in a continuous way $\Omega_1$ into $\Omega_2$; finally, again using a continuous Steiner symmetrization we may deform $\Omega_2$ into the ball $\B$. 
By composing such deformations with $(\lambda_1,T^{-1})$, we obtain three continuous paths, that can be reparametrized from $[0,1]$ to $\E$. Following the order above, we denote them by $\eta_i(\cdot)$, $i=1,2,3$, and their concatenation by $\eta(\cdot):[0,3]\to \E$. Notice that the constructed path is not unique, since the Minkowski sum is not invariant under rigid motion of sets and the sets associated to the same point of the diagram are not necessarily unique.

\begin{proposition}\label{lemmotopo} Let $\Omega_1,\Omega_2$ be convex sets of $\mathbb R^N$ of unit measure and let $\eta:[0,3]\to \E$ be the continuous closed curve constructed above. Then all the points of the plane with winding number different from zero are in the diagram.
\end{proposition}
For the definition of the winding number of a curve around a point, see, e.g. \cite{libroAC}. Roughly speaking, our result states that if $\eta$ is a simple curve, then all the points of the bounded region enclosed by $\eta$ are in $\E$.
\begin{proof}[Proof of Proposition \ref{lemmotopo}]
Assume by contradiction that there exists $(x_1,y_1)\notin \E$ such that the winding number of $\eta$ around it is $k\neq 0$. We now introduce an auxiliary function $H$ depending on two variables, $(s,t)\in [0,1]\times[0,3]$, with values in $\E$. For every $s\in [0,1]$ we define $H(s;\cdot)$ as the concatenation of three curves:
\begin{itemize}
\item[-] for $t\in [0,1]$, $H(s;\cdot)$ is the re-parametrization of $\eta_1$ from $\eta_1(0)$ to $\eta_1(1-s)$;
\item[-] for $t\in [1,2]$, $H(s;\cdot)$ is the image of the normalized Minkowski curve from $\eta_1(1-s)$ to $\eta_3(s)$;
\item[-] for $t\in [2,3]$, $H(s;\cdot)$ is the re-parametrization of $\eta_3$ from $\eta_3(s)$ to $\eta_3(1)$.
\end{itemize}
The function $H$ is continuous in both variables. Moreover, for every $s$ fixed, it defines a closed path, which is $\eta$ for $s=0$ and the constant path $\mathbb V$ for $s=1$. Therefore, $H$ is a homotopy from $\eta$ to the constant path. Since the winding number is invariant under homotopy, and since the winding number of the constant path around $(x_1,y_1)$ is 0, we find the contradiction.
\end{proof}

We underline that the key point of the previous proof is that the continuous path $\eta$ comes from a continuous deformation of a set. This is no longer true for generic closed paths, concatenating, e.g., portions of $\Gamma^+$, a suitable Minkowski curve, and a portion of $\Gamma^-$. Notice that the monotonicity of $L^+$ is unknown. 

We conclude the paragraph by showing the result of some numerical simulations in dimension 2, performed with Matlab, which give some intuition on the diagram $\E$.

\begin{figure}[h]   
{\includegraphics[height=7truecm]                        
{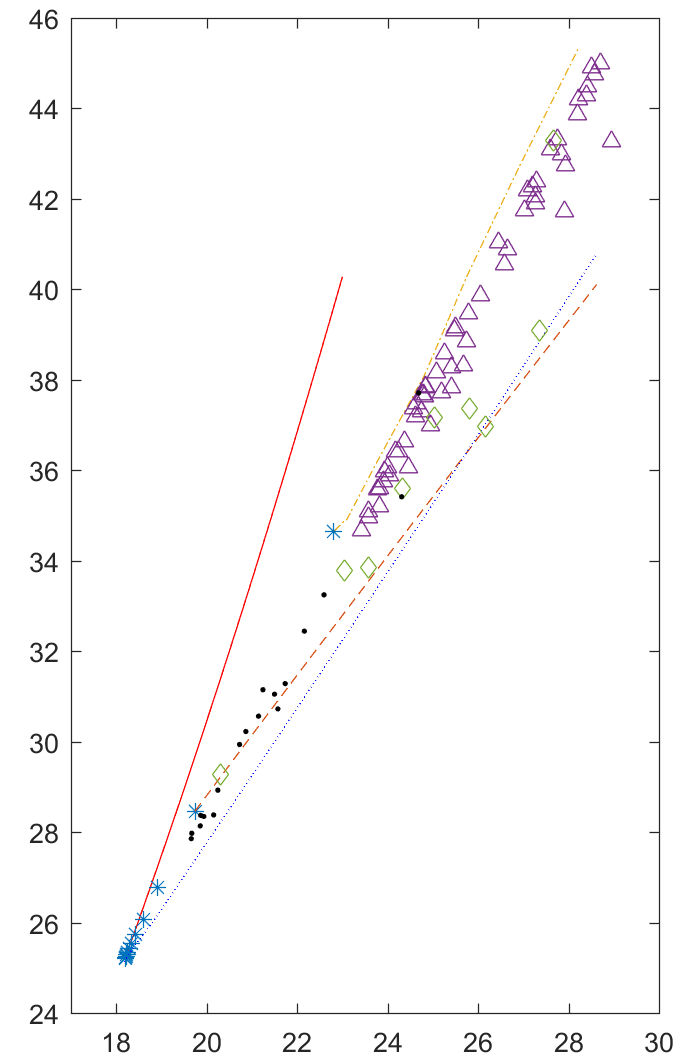}}                       \ \ \  \                         
{\includegraphics[height=7.1truecm]                         
{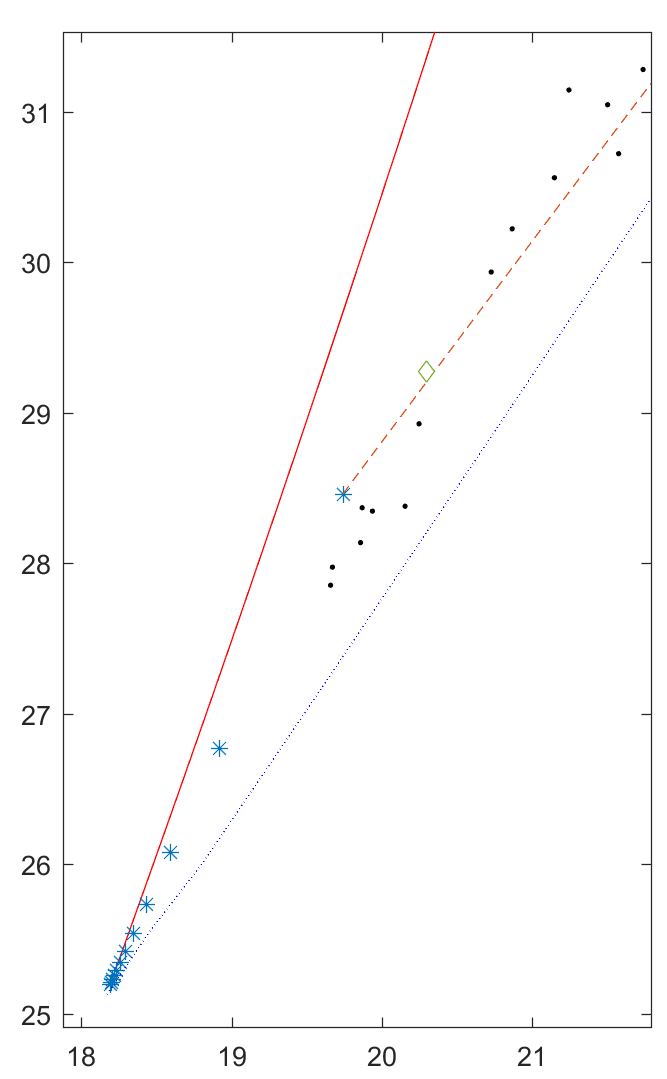}}   
\caption{{\it On the left: the continuous line is the Kohler-Jobin curve $\Gamma_\B$, the dotted line corresponds to ellipses, the dashed line to rectangles, the dotted-dashed line to isosceles triangles, the symbol $*$ is for regular polygons, $\Delta$ for random triangles, $\diamondsuit$ for random quadrilaterals, and the dots for random polygons. On the right: a zoom near the disk.}}
\end{figure}            

\subsection{Blaschke-Santal\'o diagrams on generic sets}
An interesting research line could be to remove the convexity constraint, namely to study the attainable sets  $\widetilde{\E}$ and $\widetilde{\D}$ of $(\lambda_1, T^{-1})$ among the open sets of measure equal to 1 and at most 1, respectively. Here the volume constraint is not as rigid as in the convex framework. Actually, it is easy to see that the two diagrams are essentially the same, and $\widetilde{\E}$ is dense in $\widetilde{\D}$. The Kohler-Jobin curve $\Gamma_\B$ is still an upper barrier for the diagrams, included in $\widetilde{\D}$ but not in $\widetilde{\E}$. The difficult point concerns the ``lower'' boundary. One natural question is the following: 
\begin{center}
{\bf Open problem 3:} Determine the minimal slope at $\mathbb V$ in $\widetilde{\E}$.
\end{center} 
The study carried out so far gives a partial answer: on one hand, the diagram $\widetilde\E$ is contained into the upper right quadrant delimited by $x=X(\B)$ and $y=Y(\B)$; on the other hand, it contains $\E$. Thus we immediately get $0\leq \widetilde{\gamma}^-\leq \gamma^-$. 
A natural conjecture would be that $\widetilde{\gamma}^-$ agrees with one of these two bounds. Note that the explicit value of $\gamma^-$ is not known. 
In this respect, we are able to give an upper bound for $\widetilde{\gamma}^-$ which is better (in dimension 2) than the one found in Proposition \ref{prop-valorislopes}: 
$$
\widetilde{\gamma}^-\leq \displaystyle{\frac{1}{T(\B)^2}\,\frac{w_\B(0)^2}{\varphi_\B(0)^2}}.
$$
This upper bound corresponds to the slope of the trajectory $\e \mapsto (\X(\B \setminus \overline{B}_{\e}(0)), \Y(\B \setminus \overline{B}_{\e}(0)))$ at $\e=0$, namely to a sequence of balls perforated by a vanishing smaller ball. The precise value is $4 |J_0'(j_{0,1})|^2/J_0(0)^2=4 |J_1(j_{0,1})|^2\sim 1.078$  and, as expected, since perforations break the convexity constraint, it is strictly less than the upper bound of $\gamma^-$ found in \eqref{defgammameno}, whose value is $\sim 1.4626$.

\section{Appendix}\label{sec-app}
This section is devoted to the proof of Proposition \ref{prop-ds}, namely to the computation of the second order shape derivatives of $T$ and $\lambda_1$ at $\B$ in dimension 2, with respect to deformations which  preserve convexity and keep the volume unchanged. For the formulas of shape derivatives see \cite[Chapter 5]{HP} and \cite{BFL1, BFL2,NP}. Similar computations in terms of Fourier coefficients can be found in \cite{Bogo,BHL}.

The representation \eqref{Fourier} in terms of support functions accounts for the convexity constraint. As for the volume constraint, since we perform a second order analysis, it is enough to impose that the first and second order shape derivatives of the area vanish. These imply a constraint on the Fourier coefficients. 

\begin{lemma}\label{volume} Let $V$ and $W$ be two admissible deformations in $\mathcal A$. Denote by $\alpha$ and $\beta$ be the first and second variation of the support function, defined according to \eqref{supporto}-\eqref{Fourier}. Then
\begin{equation}\label{optcond}
\int_0^{2\pi}\alpha(\theta)\de \theta=0,\quad \int_0^{2\pi} \beta(\theta)\de \theta = \frac{1}{R}\int_0^{2\pi}[\dot{\alpha}(\theta)^2 -\alpha(\theta)^2] \de \theta.
\end{equation}
\end{lemma}
\begin{proof}
By assumption, for every $\e$ small, the volume, denoted here by $\mathrm{Vol}$, is constant, namely $\mathrm{Vol}(\Omega_\e)=\mathrm{Vol}(\B)$. In particular, $\mathrm{Vol}'(\B;V)=\mathrm{Vol}''(\B;V,W)=0$. In view of the well known formulas for $\mathrm{Vol}'$ and $\mathrm{Vol}''$ (see for instance \cite[\S 5.9.3 and \S5.9.6]{HP}), we have
\begin{equation}\label{volumi}
\mathrm{Vol}'(\B;V)=\int_{\partial \B} V\cdot n \, \de \mathcal H^1 = 0, \quad \mathrm{Vol}''(\B;V,W)= \int_{\partial \B} [\kappa (V\cdot n)^2 + Z + W\cdot n]\, \de \mathcal H^1=0,
\end{equation}
where $\kappa$ denotes the mean curvature, here equal to $1/R$, and $Z$ is the following function, defined on $\partial \B$:
\begin{equation}\label{defZ}
Z\coloneqq   ( D_\Gamma n\, V_\Gamma ) \cdot V_\Gamma - 2 [\nabla_\Gamma (V\cdot n)] \cdot V_\Gamma.
\end{equation}
The subscript $\Gamma$ denotes the tangential component of a vector/operator: for a vector field $U$ and a function $f$ defined in the whole $\mathbb R^2$, there hold 
$$
U_\Gamma\coloneqq U -  (U\cdot n) n\,,\quad D_\Gamma U \coloneqq D U - (DU\, n)\otimes n\,,\quad \nabla_\Gamma f\coloneqq \nabla f - (\nabla f \cdot n) n,
$$
where $DU$ denotes the Jacobian matrix of $U$.  
Let us rewrite the boundary integrals in \eqref{volumi} in polar coordinates: in view of \eqref{supporto}, we have $V\cdot n = \alpha$, $V\cdot \tau =\dot{\alpha}$, and $W\cdot n = \beta$, so that \eqref{volumi} reads
\begin{equation}\label{quasi}
\int_0^{2\pi} \alpha(\theta)\de \theta = \int_0^{2\pi }[\alpha(\theta)^2 + R Z + R\beta] \de \theta =0.
\end{equation}
Choosing any extension of $n$, $\tau$, and $V$ to $\mathbb R^2$, we find $( D_\Gamma n\, V_\Gamma ) \cdot V_\Gamma=[\nabla_\Gamma (V\cdot n)] \cdot V_\Gamma=\dot{\alpha}^2/R$, so that
\begin{equation}\label{zeta}
Z(R \cos \theta, R \sin \theta)=- \frac{\dot{\alpha}^2}{R}.
\end{equation}
Inserting this expression in \eqref{quasi} we conclude the proof.
\end{proof}

\begin{proof}[Proof of Proposition \ref{prop-ds}] Throughout the proof, for brevity, we will omit the subscript $\B$ in the first eigenfunction and in the torsional rigidity, which will be denoted by $\varphi$ and $w$, respectively. The second order shape derivatives of $\lambda_1$ and $T$ at $\B$ are
\begin{align}
\lambda_1''(\B;V,W) & = \int_{\partial \B}  \left( - W\cdot n- Z  + \kappa (V\cdot n)^2 \right) |\partial_\nu \varphi|^2 \de \mathcal H^1 + 2 \int_{\partial \B}\z \partial_\nu \z  \de \mathcal H^1,   \label{Lsec}
\\
T''(\B;V,W) & = \int_{\partial \B} \left[    \left( W\cdot n + Z - \kappa (V\cdot n)^2\right)|\partial_\nu w|^2  + 2  (V\cdot n)^2 |\partial_\nu w| \right]\de \mathcal H^1 
- 2 \int_{\partial \B} v \partial_\nu v \, \de \mathcal H^1,   \label{Tsec}
\end{align}
where $\kappa$ is the curvature, $Z$ is the function introduced in \eqref{defZ}, and $\z$ and $v$ solve
\begin{equation}\label{systems}
\left\{
\begin{array}{lll}
-\Delta \z = \lambda_1(\B) \z  - \varphi \int_{\partial \B} |\partial_\nu \varphi|^2 V\cdot n\, \de \mathcal H^1 \quad & \hbox{in }\B
\\
\z = -(V\cdot n) \partial_\nu \varphi \quad  & \hbox{on }\partial \B
\\
\int_{\B} \z \varphi=0
\end{array}
\right.\qquad \left\{
\begin{array}{lll}
\Delta v = 0\quad & \hbox{in }\B
\\
v = -(V\cdot n) \partial_\nu w \quad  & \hbox{on }\partial \B.
\end{array}
\right.
\end{equation}
We recall that the torsion function of the disk $\B$ is $w=(R^2- |x|^2)/4$ so that, on the boundary, we have $|\partial_\nu w| =R/2$. Similarly, since the first eigenfunction of the Dirichlet Laplacian, normalized in $L^2$, is $
\varphi= J_0(j_{0,1} |x|/R )/|J_0'(j_{0,1})|$, we have $|\partial_\nu \varphi |=j_{0,1}/R$ on the boundary. Let us perform the change of variables in polar coordinates in the integrals above. Using the fact that $\kappa=1/R$, writing $Z$ as in \eqref{zeta}, recalling the expression \eqref{supporto} of $V$ on $\partial \B$ in terms of $\alpha$, and exploiting the conditions \eqref{optcond} on $\alpha$ and $\beta$, we obtain a first simplification:
\begin{align}
\lambda_1''(\B;V,W) & = \frac{2 j_{0,1}^2}{R^2} \int_0^{2\pi} \alpha^2 \de \theta + 2 R \int_0^{2\pi}\z \partial_\nu \z  \, \de \theta,\label{w1}
\\
T''(\B;V,W)& = \frac{R^2}{2} \int_0^{2\pi} \alpha^2\de \theta 
- 2 R \int_0^{2\pi} v \partial_\nu v \, \de \theta.\label{v1}
\end{align}
Let us now determine $\z$ in terms of $\alpha$ and of its Fourier coefficients $a_m$ and $b_m$. First, we notice that, in view of the condition $\int \alpha=0$ in \eqref{optcond}, the PDE solved by $\z$ is $-\Delta \z = \lambda_1(\B)\z$. Therefore, it is natural to look for $\z$ as a linear combination (possibly a series) of the eigenfunctions $J_m(j_{0,1} \rho/R)\cos(m\theta)$ and $J_m(j_{0,1} \rho/R)\sin(m\theta)$ associated to the eigenvalue $\lambda_1(\B)=j_{0,1}^2/R^2$, namely $\z(\rho,\theta)=\sum_{m\geq 0} [A_m \cos(m\theta) + B_m \sin(m \theta)] J_m(j_{0,1}\rho/R)$. The orthogonality condition between $\z$ and the radial function $\varphi$ gives $A_0=0$. Imposing the boundary condition $\z(R,\theta)=j_{0,1}\alpha(\theta)/R$, we get
$$
A_m=\frac{j_{0,1} a_m}{ R J_m (j_{0,1})}\,,\quad B_m=\frac{ j_{0,1}  b_m}{R J_m (j_{0,1})}\,,\quad \forall m\geq1\,.
$$
A direct computation leads to 
\begin{equation}\label{w2}
 \int_{0}^{2\pi} \z \partial_\nu \z \de \theta =   \frac{\pi j_{0,1}^3}{R^3} \sum_{m\geq 1 }\frac{J_m'(j_{0,1})}{J_m(j_{0,1})}(a_m^2 + b_m^2).
\end{equation}
By combining \eqref{w1} and \eqref{w2}, recalling that $\int_0^{2\pi} \alpha^2 = \pi \sum_{m\geq 1}(a_m^2 + b_m^2)$ and using $j_{0,1} J_1'(j_{0,1})=- J_1(j_{0,1})$, we get
$$
\lambda_1''(\B;V,W)  =  \frac{2\pi  j_{0,1}^2}{R^2}  \sum_{m\geq 2}\left[ \left( 1 +  j_{0,1}\frac{J_m'(j_{0,1})}{J_m(j_{0,1})} \right) (a_m^2+b_m^2)\right].
$$
Following the same procedure, we may derive $v$ as a function of $a_m$ and $b_m$. Formally, $v$ can be searched as the infinite sum of harmonic functions, namely $v(\rho,\theta)=\sum_{m\geq 0} [C_m \cos(m\theta) + D_m \sin(m\theta)] \rho^m$. Imposing the boundary condition we obtain
$$
C_0=D_0=0\,,\quad C_m  = \frac{a_m}{2 R^{m-1}} \,,\quad D_m =  \frac{ b_m}{2R^{m-1}}\,,\quad \forall m\geq1\,. 
$$
In particular, 
$$
\int_0^{2\pi} v \partial_\nu v \de \theta = \frac{\pi R}{4} \sum_{m\geq 1} m (a_m^2 + b_m^2),
$$
and \eqref{v1} reads
$$
T''(\B;V,W)   = \frac{\pi R^2}{2} \sum_{m\geq 2} \left[(1-m)(a_m^2 + b_m^2)   \right].
$$
This concludes the proof.
\end{proof}

\begin{remark}\label{lastremark}
At first sight, the equalities \eqref{L''}-\eqref{T''} might seem surprising, since $W$ apparently does not play any role. Actually, as it is clear from the formulas used in the previous proof, in the second order shape derivatives only the normal component of $W$ appears, averaged with $|\nabla w_\B|^2$ or $|\nabla \varphi_\B|^2$ on the boundary. Since both norms of the gradients are constant, the relevant quantity is the average of $W\cdot n$. The average is nothing but $\int \beta = c_0$, which in turn can be written in terms of $\alpha$ or $\{a_m, b_m\}$, as we have proved in Lemma \ref{optcond} and rephrased in \eqref{lastformula}.
\end{remark}

\medskip

\noindent{\bf Acknowledgements.} The authors are grateful to A. Henrot for having suggested the problem, and thank G. Buttazzo, I. Ftouhi, A. Henrot, and J. Lamboley for the fruitful discussions.
The authors are members of the Gruppo Nazionale per l'Analisi Matematica, la Probabilit\`a e le loro Applicazioni (GNAMPA) of the Istituto Nazionale di Alta Matematica (INdAM). 

The work of I.L. was partially supported by the project ANR-18-CE40-0013 SHAPO financed by the French Agence Nationale de la Recherche (ANR), and by the Ypatia Laboratory of Mathematical Sciences (LIA LYSM AMU CNRS ECM INdAM). I.L. acknowledges the Math Department of the University of Pisa for the hospitality. D.Z. acknowledges support of the Research Project INdAM for Young Researchers (Starting Grant) \emph{Optimal Shapes in Boundary Value Problems} and of the INdAM - GNAMPA Project 2018 \emph{Ottimizzazione Geometrica e Spettrale}.

\end{document}